\documentclass[12pt]{amsart}
\usepackage{amssymb,amsfonts,ifthen}
\usepackage[a4paper]{geometry}
\usepackage[colorlinks]{hyperref}
\usepackage[capitalize]{cleveref}

\crefname{def}{Definition}{Definitions}
\Crefname{def}{Definition}{Definitions}
\creflabelformat{def}{~\upshape(#2#1#3)}

\crefname{ineq}{Ineq.}{Ineqs.}
\Crefname{ineq}{Inequality}{Inequalities}
\creflabelformat{ineq}{~\upshape(#2#1#3)}

\crefname{rec}{Recursion}{Recursions}
\Crefname{rec}{Recursion}{Recursions}
\creflabelformat{rec}{~\upshape(#2#1#3)}

\crefname{rl}{Relation}{Relations}
\crefname{rl}{Relation}{Relations}
\creflabelformat{rl}{~\upshape(#2#1#3)}

\crefname{fml}{Formula}{Formulas}
\crefname{fml}{Formula}{Formulas}
\creflabelformat{fml}{~\upshape(#2#1#3)}

\crefname{rst}{Result}{Results}
\crefname{rst}{Result}{Results}
\creflabelformat{rst}{~\upshape(#2#1#3)}

\crefrangeformat{equation}{Eqs.~(#3#1#4) ---~(#5#2#6)}
 
\numberwithin{equation}{section}
\numberwithin{figure}{section}
\numberwithin{table}{section}

\newtheorem{thm}{Theorem}[section]
\newtheorem{cor}[thm]{Corollary}
\newtheorem{lem}[thm]{Lemma}

\newtheorem{prop}[thm]{Proposition}

\theoremstyle{definition}
\newtheorem{dfn}[thm]{Definition}
\newtheorem{ntt}[thm]{Notation}
\newtheorem{eg}[thm]{Example}

\theoremstyle{remark}
\newtheorem*{rmk}{Remark}

\def\fl#1{\lfloor{#1}\rfloor}
\def\bgg#1{\biggl({#1}\biggr)}
\def\all#1{\qquad\textrm{for all $#1$}}
\def\tW{\tilde{W}}
\def\tR{\tilde{R}}
\def\txi{\tilde{\xi}}
\def\tD{\tilde{\Delta}}
\def\tg{\tilde{g}}
\def\ta{\tilde{a}}
\def\tb{\tilde{b}}
\def\Z{\mathbb{Z}}
\def\Q{\mathbb{Q}}
\def\R{\mathbb{R}}
\def\<{\;<\;}
\def\={\;=\;}
\def\>{\;>\;}
\def\pv{\mathrm{pv}}
\def\eqrl{\quad\iff\quad}
\newcommand{\rmand}{\quad\text{ and }\quad}
\newcommand{\oz}{$(1,0)$-sequence}

\begin{document}

\title[]{Root geometry of polynomial sequences II: \\
Type $(1,0)$}

\author[J.L. Gross, T. Mansour, T.W. Tucker, and D.G.L. Wang]{Jonathan L. Gross
}
\address{
Department of Computer Science  \\
Columbia University, New York, NY 10027, USA; \quad
email: gross@cs.columbia.edu
}
\author[]{Toufik Mansour
}
\address{
Department of Mathematics  \\
University of Haifa, 31905 Haifa, Israel;  \quad
email: tmansour@univ.haifa.ac.il}
\author[]{Thomas W. Tucker
}
\address{
Department of Mathematics  \\
Colgate University, Hamilton, NY 13346, USA; \quad
email: ttucker@colgate.edu
}
\author[]{David G.L. Wang
}
\address{
School of Mathematics and Statistics  \\
Beijing Institute of Technology, 102488 Beijing, P. R. China;  \quad
email: david.combin@gmail.com}


\bigskip

\begin{abstract}
We consider the sequence of polynomials $W_n(x)$ defined by the recursion $W_n(x)=(ax+b)W_{n-1}(x)+dW_{n-2}(x)$, with initial values $W_0(x)=1$ and $W_1(x)=t(x-r)$, where $a,b,d,t,r$ are real numbers, $a,t>0$, and $d<0$. We show that every polynomial $W_n(x)$ is distinct-real-rooted, and that the roots of the polynomial $W_n(x)$ interlace the roots of the polynomial $W_{n-1}(x)$. We find that, as $n\to\infty$, the sequence of smallest roots of the polynomials $W_n(x)$ converges decreasingly to a real number, and that the sequence of largest roots converges increasingly to a real number. Moreover, by using the Dirichlet approximation theorem, we prove that there is a number to which, for every positive integer $i\ge2$, the sequence of $i$th smallest roots of the polynomials $W_n(x)$ converges. Similarly, there is a number to which, for every positive integer $i\ge2$, the sequence of $i$th largest roots of the polynomials $W_n(x)$ converges. It turns out that these two convergence points are independent of the numbers $t$ and $r$, as well as $i$.  We derive explicit expressions for these four limit points, and we determine completely when some of these limit points coincide. 
\end{abstract}

\maketitle

\tableofcontents

\section{\large Introduction}\label{sec:intro}  

In \cite{GMTW15X}, we initiated study of the root geometry of 
a recursively defined sequence $\{W_n(x)\}_{n\ge0}$ of polynomials. 
Originally motivated by the LCGD conjecture from topological graph theory,
which says that the genus polynomial of every graph is log-concave, 
we are led to study the log-concavity of the polynomials in $\{W_n(x)\}_{n\ge0}$; 
see \cite{GF87, GMTW14X-claw, GMTW15, GRT89}. 
Extending our perspective from the arithmetic property to algebraic structures, 
and following the spirit of Gian-Carlo Rota, 
we study the distribution of the zero-sets of the polynomials $\{W_n(x)\}_{n\ge0}$. 
Such a distribution is called the {\em root geometry} of the sequence. 
General information for the root geometry of collections of polynomials 
can be found in Marden \cite{Mar85B}, Rahman and Schmeisser \cite{RS02B}, 
and Prasolov \cite{Pra10B}.

As encountered naturally in biology, economics, computer sciences, combinatorics, 
and other branches of mathematics, recurrence relations are among the most familiar objects 
that mathematicians work on. Proceeding systematically, 
we \cite{GMTW15X} introduced the concept of a {\em polynomial sequence of type $(k,l)$}, 
which satisfies a recursion of the form
\[
W_n(x) \= A(x)W_{n-1}(x)+B(x)W_{n-2}(x)
\]
for $n\ge 2$, where $A(x)$ and $B(x)$ are polynomials of degrees $k$ and $l$,  respectively, 
and where $W_0(x)$ is a constant and $W_1(x)$ a linear polynomial. 
For the first non-trivial case, namely, sequences of type $(0,1)$, 
we showed that subject to some general conditions that hold for all graph genus polynomials, 
every polynomial is distinct-real-rooted and that the zero-sets 
of the polynomials $W_{n+1}(x)$ and $W_n(x)$ are interlacing. 
We also found a sharp bound for the union of all zero-sets over~$n$.

This paper continues the study of root geometry of recursive polynomials, 
this time for the second non-trivial case, namely, those of type $(1,0)$.  
We will confirm the distinct-real-rootedness of such polynomials subject to similar general conditions, 
and we determine the best bound for the union of all zero-sets. 
By using Dirichlet's approximation theorem, 
we calculate some limit points of the union. 
Classical examples for this kind of recursive polynomials include 
Chebyshev polynomials of the first and second kinds. 

This paper is organized as follows. 
\Cref{sec:main:10} contains the root geometry of polynomial sequences of type $(1,0)$ 
as our main result, as well as some applications. 
We show a particular case for \cref{thm:10} in \cref{sec:pf1:10}, namely, \cref{thm:10:normalized}, using
which we complete the proof of \cref{thm:10} in \cref{sec:pf2:10}.

\bigskip
\section{\large Main Theorem}\label{sec:main:10} 

In this section we describe some information for the root geometry of a polynomial sequence of type $(1,0)$ 
as the main result, and we present some applications for illustration.
\medskip

\subsection{Main result}  

We begin with definitions and notation needed or the statement of \cref{thm:10}, which is our main result.
\smallskip

\begin{dfn}
A polynomial is said to be \emph{distinct-real-rooted} if all its roots are distinct and real. 
The set of all its roots is called the \emph{zero-set} of a polynomial.
\end{dfn}
\smallskip

\begin{dfn}\label{def:itl:10}
Let $X=\{x_1,x_2,\ldots,x_n\}$ and $Y=\{y_1,y_2,\ldots,y_{n-1}\}$ be ordered sets of real numbers. 
We say that  \emph{$X$ interlaces $Y$}, denoted $X\bowtie Y$, if
\[
x_1\<y_1\<x_2\<y_2\<\cdots\<x_{n-1}\<y_{n-1}\<x_n.
\]
A special case is that a singleton set $\{x\}$ interlaces the empty set. 
\end{dfn}

\begin{ntt}
Let $\{x_n\}_{n\ge 0}$ be a sequence of real numbers.
We write $x_n\searrow x$ if the sequence converges to the number~$x$ decreasingly as $n\to\infty$,
and we write $x_n\nearrow x$ if it converges to~$x$ increasingly.
\end{ntt}
\smallskip

Our main result is as follows.

\begin{thm}\label{thm:10}
Let $\{W_n(x)\}_{n\ge0}$ be the polynomial sequence defined recursively by
\begin{equation}\label[rec]{rec:10}
W_n(x) \= (ax+b)W_{n-1}(x)+d\cdotp W_{n-2}(x),
\end{equation}
with $W_0(x)=1$ and $W_1(x)=t(x-r)$, where $a,t>0$, $d<0$, and $b,r\in\R$. 
Then the polynomial~$W_n(x)$ has degree~$n$ and is distinct-real-rooted. 
Also, we define
\begin{align}
\label[def]{def:r+-:10} r^\pm &\= -{b\over a}\pm{|a-2t|\sqrt{-d}\over at},\\[3pt]
\label[def]{def:xd+-:10} x_\Delta^\pm &\= {-b\pm2\sqrt{-d}\over a}, \rmand \\[5pt]
\label[def]{def:xi+-:10} \xi^\pm &\= 
\begin{cases}
\displaystyle {t(ar-b-2rt)\pm\sqrt{t^2(ar+b)^2-4dt(a-t)}\over 2t(a-t)}, \quad \text{if $a\ne t$};\\[5pt]
\displaystyle r-{d\over a(ar+b)}, \qquad \text{if $a=t$ and $ar+b\ne0$}.
\end{cases}
\end{align}
Denote the zero-set of the polynomial $W_n(x)$ by 
$R_n=\{\xi_{n,1},\,\xi_{n,2},\,\ldots,\,\xi_{n,n}\}$ in increasing order. 
Then we have the interlacing property $R_{n+1}\bowtie R_{n}$, and the limits 
\begin{equation}\label[rst]{lim:10}
\xi_{n,j}\searrow x_\Delta^-
\rmand
\xi_{n,\,n+1-j}\nearrow x_\Delta^+,
\all{j\ge 2}.
\end{equation}
Moreover, we have the following cases of convergence results.
\begin{itemize}
\smallskip\item[(i)]
If $r\in[\,r^-,\,r^+]$ and $a\le 2t$, then $\xi_{n,1}\searrow x_\Delta^-$ and $\xi_{n,n}\nearrow x_\Delta^+$.
\smallskip\item[(ii)]
If $r\in(r^-,\,r^+)$ and $a>2t$, then $\xi_{n,1}\searrow \xi^-$ and $\xi_{n,n}\nearrow \xi^+$.
\smallskip\item[(iii)]
If $r=r^-$ and $a>2t$, or $r<r^-$, then $\xi_{n,1}\searrow \xi^-$ and $\xi_{n,n}\nearrow x_\Delta^+$. 
\smallskip\item[(iv)]
If $r=r^+$ and $a>2t$, or $r>r^+$, then $\xi_{n,1}\searrow x_\Delta^-$ and $\xi_{n,n}\nearrow \xi^+$.
\end{itemize}
\end{thm}

Recall from~\cite{GMTW15X} that the sequence of largest roots of a polynomial sequence of type $(0,1)$ 
converges to a real number in an oscillating manner.
In contrast, for any polynomial sequence of type $(1,0)$, 
the sequence of $i$th smallest roots converges decreasingly as $n\to\infty$, 
and the sequence of $i$th largest roots converges increasingly.

\begin{rmk} The numbers $\xi^\pm$ are not defined when $a=t$ and $ar+b=0$. 
This is an instance of Case (i) of \cref{thm:10}. 
\end{rmk}

To give a proof of \cref{thm:10}, we state its normalized version as \cref{thm:10:normalized}.
The following notion of $(1,0)$-sequence of polynomials is the key object we will study;
see \cref{sec:pf1:10}. As will be seen, \cref{thm:10:normalized} implies \cref{thm:10}.

\begin{dfn}\label{def:W:10}
Let $\{W_n(x)\}_{n\ge0}$ be the polynomial sequence defined recursively by
\[
W_n(x)=(ax+b)W_{n-1}(x)+dW_{n-2}(x),
\]
with $W_0(x)=1$ and $W_1(x)=x$,
where $a>0$, $b\ge 0$ and $d<0$.
In this context, we say $\{W_n(x)\}_{n\ge0}$ is a {\em normalized $(1,0)$-sequence of polynomials},
or a {\em $(1,0)$-sequence} for short.
The following theorem concerns the particular case of \cref{thm:10} for which $t=1$, $r=0$, and $b\ge 0$. 
\end{dfn}
\smallskip

\begin{thm}\label{thm:10:normalized}
Let $\{W_n(x)\}_{n\ge0}$ be a \oz.
Then the polynomial~$W_n(x)$ has degree~$n$ and is distinct-real-rooted. 
Let 
\begin{align}
b_0&\=|a-2|\sqrt{-d} \rmand \label[def]{def:b0:10}\\[5pt]
x_g^\pm&\=
\begin{cases}
\displaystyle {-b\pm\sqrt{b^2-4d(a-1)}\over 2(a-1)},&\text{if $a\ne 1$};\\[5pt]
\displaystyle -{d\over b},&\text{if $a=1$ and $b\ne0$}.
\end{cases}\label[def]{def:xg+-:10}
\end{align}
Denote the zero-set of the polynomial~$W_n(x)$ by~$R_n=\{\xi_{n,1},\,\xi_{n,2},\,\ldots,\,\xi_{n,n}\}$ in increasing order. 
Then we have $R_{n+1}\bowtie R_{n}$ and \cref{lim:10}.
Moreover, we have the following.
\begin{itemize}
\smallskip\item[(i)]
If $a\le 2$ and $b\le b_0$, then $\xi_{n,1}\searrow x_\Delta^-$ and $\xi_{n,n}\nearrow x_\Delta^+$.
\smallskip\item[(ii)]
If $a>2$ and $b<b_0$, then $\xi_{n,1}\searrow x_g^-$ and $\xi_{n,n}\nearrow x_g^+$.
\smallskip\item[(iii)]
Otherwise, we have $\xi_{n,1}\searrow x_\Delta^-$ and $\xi_{n,n}\nearrow x_g^+$.
\end{itemize}
\end{thm}
\smallskip

We note that the limit points $x_\Delta^\pm$ are independent of the numbers $t$ and $r$. 

\subsection{Some examples}  

In this subsection, we present several applications of our results.

\begin{eg}\label{eg:Chebyshev}  
Let $W_n(x)$ be the polynomials defined by the recursion 
\[
W_n(x) \= 2xW_{n-1}(x)-W_{n-2}(x),
\] 
with $W_0(x)=1$ and  $W_1(x)$ linear.
By \cref{thm:10}, every polynomial $W_n(x)$ is of degree~$n$ and distinct-real-rooted. Moreover, we have
\[
\xi_{n,j}\searrow-1
\rmand
\xi_{n,\,n+1-j}\nearrow 1
\all{j\ge2}.
\]
If $W_1(x)=x$, then \cref{thm:10}~(i) gives that 
\begin{equation}\label[rst]{eg1}
\xi_{n,1}\searrow -1
\rmand
\xi_{n,n}\nearrow 1.
\end{equation}
In fact, the polynomials $W_n(x)$ are {\em Chebyshev polynomials of the first kind}, whose zero-sets are known to be
\[
R_n=\biggl\{\cos{(2j-1)\pi\over 2n}\colon j\in[n]\biggr\}.
\]
If $W_1(x)=2x$, then \cref{thm:10} (i) also yields \cref{eg1}.
In this case, the polynomials $W_n(x)$ are {\em Chebyshev polynomials of the second kind}, whose zero-sets are known to be 
\[
R_n=\biggl\{\cos{j\pi\over n+1}\colon j\in[n]\biggr\}.
\]
If $W_1(x)=x/2$, then \cref{thm:10} (ii) implies 
\[
\xi_{n,1}\searrow-2/\sqrt{3}
\rmand
\xi_{n,n}\nearrow2/\sqrt{3}.
\]
If $W_1(x)=x+1$, then \cref{thm:10}~(iii) implies that 
\[
\xi_{n,1}\searrow -\sqrt{2}
\rmand
\xi_{n,n}\nearrow 1.
\]
If $W_1(x)=x-2$, then \cref{thm:10}~(iv) implies
\[
\xi_{n,1}\searrow -1
\rmand
\xi_{n,n}\nearrow\sqrt{5}.
\]
\end{eg}
\smallskip

\begin{eg}\label{eg:A130777}  
Suppose that $W_0(x)=1$, $W_1(x)=x-1$, and 
$$W_n(x)\=xW_{n-1}(x)-W_{n-2}(x).$$
\Cref{thm:10}~(i) implies that $\xi_{n,j}\searrow -2$ and $\xi_{n,\,n+1-j}\nearrow2$ for all $j\ge1$.  
In fact, the zero-set $R_n$ of the polynomial $W_n(x)$ is known~\cite[A130777]{OEIS} to be
\[
R_n=\biggl\{-2\cos{2j\pi\over 2n+1}\colon j\in[n]\biggr\}.
\]
\end{eg}
\smallskip

\begin{eg}\label{eg:A101950}  
Suppose that $W_0(x)=1$, $W_1(x)=x+1$, and 
$$W_n(x)\=(x+1)W_{n-1}(x)-W_{n-2}(x).$$
\Cref{thm:10}~(i) implies that $\xi_{n,j}\searrow -3$ and $\xi_{n,\,n+1-j}\nearrow1$ for all $j\ge1$; 
see~\cite[A101950]{OEIS}. 
\end{eg}
\smallskip

\begin{eg}\label{eg:A104562}  
Suppose we have  $W_0(x)=1$, $W_1(x)=x-1$, and 
$$W_n(x)\=(x-1)W_{n-1}(x)-W_{n-2}(x).$$
From \Cref{thm:10} (i) we infer that $\xi_{n,j}\searrow -1$ 
and $\xi_{n,\,n+1-j}\nearrow3$ for all $j\ge1$; see~\cite[A104562]{OEIS}.
\end{eg}
\smallskip

We will prove \cref{thm:10} in the next two sections.  
In \cref{sec:pf1:10}, we show it for a particular kind of polynomial sequence 
of type $(1,0)$, in which $W_1(x)=x$ and $b\ge0$.
In \cref{sec:pf2:10}, we complete the proof of \cref{thm:10} 
by translating and scaling the roots, so as to drop these two conditions.

\bigskip
\section{\large Proof of \cref{thm:10:normalized}}\label{sec:pf1:10} 

We start by determining the degree and the leading coefficient of every polynomial in a $(1,0)$-sequence.

\begin{lem}\label{lem:deg:10}
Let $\{W_n(x)\}_{n\ge0}$ be a \oz.
Then every polynomial $W_n(x)$ has degree $n$ and leading coefficient~$a^{n-1}$. 
\end{lem}

\begin{proof}
This lemma follows from consideration of the contributions of 
the summands $(ax+b) W_{n-1}(x)$ and $d\,W_{n-2}(x)$ 
to the degree and to the leading coefficient of the polynomial $W_n(x)$.
\end{proof}

The next lemma is a cornerstone for studying the root geometry 
of polynomials defined by recursions of order~$2$; see~\cite{GMTW15X}.

\begin{lem}\label{lem:00}
Let $A,B\in\mathbb{R}$ such that $A\ne0$.  
Define $W_n=AW_{n-1}+BW_{n-2}$ recursively, with $W_0=1$ and with some number $W_1$.
Writing $\Delta=A^2+4B$ and $g^\pm=(2W_1-A\pm\sqrt{\Delta})/2$,
we have
\[
W_n\=\begin{cases}
{\displaystyle \bgg{1+{n(2W_1-A)\over A}}\bgg{{A\over 2}}^n},&\textrm{ if $\Delta=0$};\\[10pt]
{\displaystyle {g^+(A+\sqrt{\Delta}\,)^n-g^-(A-\sqrt{\Delta}\,)^n
\over 2^{n}\sqrt{\Delta}}}
,&\textrm{ if $\Delta\neq0$}.
\end{cases}
\]
In particular, 
if the complex number $A+\sqrt{\Delta}$ has the polar representation $Re^{i\theta}$, 
then we have
\[
W_n=\bgg{R\over2}^n\bgg{\cos{n\theta}+{(2W_1-A)\sin{n\theta}\over \sqrt{-\Delta}}},
\qquad\textrm{if $\Delta<0$}.
\]
\end{lem}

\medskip
\subsection{The distinct-real-rootedness}\label{sec:RR}

In this subsection we prove the distinct-real-rootedness of every polynomial in a \oz.
and we derive a bound for the union of the zero-sets of these polynomials.  
Both of them will be shown by applying an interlacing criterion, which we now develop.

\begin{ntt}
For any integers $m\le n$, we denote the set $\{m,\,m+1,\ldots,\,n\}$ by $[m,n]$.  
When $m=1$, we denote the set $[1,n]$ by~$[n]$.  
All sets mentioned in this paper are ordered sets, 
whose elements are arranged in increasing order.
\end{ntt}

\begin{lem}\label{lem:ctl:10}
Let $g(x)$ be a polynomial with zero-set $Y$.  
Let $X = \{x_1,x_2,\ldots,x_{m+1}\}$ be a set such that $X\bowtie Y$.  
And let $\alpha$ and $\beta$ be numbers such that 
$$\alpha<x_1<x_2<\cdots<x_{m+1}<\beta.$$  
Then for every $i\in[m+1]$, we have 
\begin{align}
\label[ineq]{ineq:itl:m:alpha}
g(\alpha)g(x_i)(-1)^{i}&\<0,\rmand \\[3pt]
\label[ineq]{ineq:itl:m:beta}
g(x_i)g(\beta)(-1)^{m-i}&\<0.
\end{align}
\end{lem}

\begin{proof}
By the premise $X\bowtie Y$, the polynomial $g(x)$ has no roots less than~$x_1$.
In particular, no root lies in the interval $(\alpha,x_1)$. 
Hence, from the intermediate value theorem, we infer that 
\begin{equation}\label[ineq]{pf:crt1:10}
g(\alpha)g(x_1)\>0.
\end{equation}
This confirms \cref{ineq:itl:m:alpha} for $i=1$.

From the premise $X\bowtie Y$, 
we also know that the polynomial $g(x)$ has exactly one root 
in the interval $(x_{i-1},\,x_i)$ for each integer $i\in[2,\,m+1]$.  
Here, the intermediate value theorem implies  
\begin{align*}
-g(x_1)g(x_2)&\>0\qquad(i=2),\\
-g(x_2)g(x_3)&\>0\qquad(i=3),\\
&\quad\vdots\\
-g(x_m)g(x_{m+1})&\>0\qquad(i=m+1).
\end{align*}
Multiplying the first $i-1$ of these inequalities yields
$g(x_1)g(x_{i})(-1)^{i-1}>0$, or equivalently, 
\begin{equation} \label[ineq]{ineq:<0}
g(x_1)g(x_{i})(-1)^{i}\<0.
\end{equation}
Multiplying \cref{ineq:<0} by \cref{pf:crt1:10} yields \cref{ineq:itl:m:alpha} for $i\in[2,\,m+1]$.  
\Cref{ineq:itl:m:beta} can be obtained similarly. 
\end{proof}

The next lemma provides a way to bound the set of roots of the polynomials in a \oz.
It serves as the induction step of the theorem that follows. 

\begin{lem}\label{lem:crt:itl:10}
Let $\{W_n(x)\}_{n\ge0}$ be a \oz. Let $m\ge0$ and $\alpha,\beta\in\R$.
Let $R_m$ denote the zero-set of the polynomial $W_m(x)$.  Suppose that
\begin{align}
\label[ineq]{cond:itl:alpha:10}W_{m} (\alpha)W_{m+2}(\alpha)&\>0 \rmand\\
\label[ineq]{cond:itl:beta:10}W_{m} (\beta)W_{m+2}(\beta)&\>0,
\end{align}
and also that  
\begin{equation}\label{prm:itl:10}
|R_{m+1}|=m+1,\qquad
R_{m+1}\subset(\alpha,\beta),\rmand
R_{m+1}\bowtie R_{m}.  
\end{equation}
Then we have 
\begin{equation}\label{rst:itl:10}
|R_{m+2}|=m+2,\qquad
R_{m+2}\subset(\alpha,\beta),
\rmand
R_{m+2}\bowtie R_{m+1}.
\end{equation}
\end{lem}

\begin{proof}
Since $|R_{m+1}|=m+1$ and $R_{m+1}\subset(\alpha,\beta)$, we can write 
$$R_{m+1}=\{x_1,x_2,\ldots,x_{m+1}\},$$ 
where $\alpha<x_1<x_2<\cdots<x_{m+1}<\beta$.
Since $R_{m+1}\bowtie R_m$,
we can infer from \cref{lem:ctl:10} that
\begin{align}
\label[ineq]{ineq:alpha:ui:10}
W_m(\alpha)W_m(x_i)(-1)^i&\<0,\rmand\\
\label[ineq]{ineq:beta:ui:10} 
W_m(x_i)W_m(\beta)(-1)^{m-i}&\<0,
\end{align}
for every $i\in[m+1]$.
On the other hand, setting $x=x_i$ in \cref{rec:10}, we see that
\begin{align*}
W_{m+2}(x_{i}) &\= (ax_i+b)W_{m+1}(x_i)+d\cdotp W_{m}(x_{i}) \\
	&\= d\cdotp W_{m}(x_{i}) \qquad\text{(since $x_i\in R_{m+1}$),}
\end{align*}
or equivalently, 
\begin{equation}\label{pf:eq1}
W_m(x_{i})\=W_{m+2}(x_{i})/d.
\end{equation}

Since $d<0$, substituting \cref{pf:eq1} into \cref{ineq:alpha:ui:10}, we see that
$$W_m(\alpha)W_{m+2}(x_{i})(-1)^{i}>0.$$
Multiplying this by \cref{cond:itl:alpha:10} of the premises and canceling the square (which is positive), 
we determine that
\begin{equation} \label[ineq]{ineq:all_i}
W_{m+2}(\alpha)W_{m+2}(x_{i})(-1)^{i}\>0\all{i\in[m+1]}.
\end{equation}
\Cref{cond:itl:alpha:10} implies that $W_{m+2}(\alpha)\ne0$.
\smallskip

Now define $x_0=\alpha$. Then \cref{ineq:all_i} holds for $i=0$ trivially. 
Consequently, we can replace the index~$i$ in the above inequality by $i-1$ for each integer $i\in[m+1]$, 
which gives that
\begin{equation} \label[ineq]{ineq:m+2}
W_{m+2}(\alpha)W_{m+2}(x_{i-1})(-1)^{i-1}\>0.
\end{equation}
Multiplying \cref{ineq:all_i,ineq:m+2} produces
\[
W_{m+2}(x_{i-1})W_{m+2}(x_{i})\<0.
\]
Therefore, by the intermediate value theorem,
the polynomial $W_{m+2}(x)$ has a root in the interval $(x_{i-1},\,x_{i})$.
Let $z_{i}$ be such a root for each $i\in[m+1]$. Then we have the iterated inequality 
\begin{equation}\label[rl]{pf:ordering2:10}
\alpha\<z_1\<x_1\<z_2\<x_2\<\cdots\<z_m\<x_m\<z_{m+1}\<x_{m+1}.
\end{equation}

On the other hand, 
substituting \cref{pf:eq1} into \cref{ineq:beta:ui:10}
gives 
$$W_{m+2}(x_i)W_{m}(\beta)(-1)^{m-i}>0.$$
Multiplying this by \cref{cond:itl:beta:10} (and canceling the square), we find that
\[
W_{m+2}(x_{i})W_{m+2}(\beta)(-1)^{m-i}\>0.
\]
When $i=m+1$, this latter inequality becomes
\[
W_{m+2}(x_{m+1})W_{m+2}(\beta)\<0.
\]
Again by the intermediate value theorem,
the polynomial $W_{m+2}(x)$ has a root $z_{m+2}$ in the interval $(x_{m+1},\,\beta)$.  
Combining with \cref{pf:ordering2:10}, we obtain 
\[
\alpha\<z_1\<x_1\<z_2\<x_2\<\cdots\<z_{m+1}\<x_{m+1}\<z_{m+2}\<\beta.
\]

We now define $T=\{z_1,z_2,\ldots,z_{m+2}\}$. The ordering immediately above implies that
\[
|T|=m+2,\qquad
T\subset (\alpha,\beta),  
\rmand
T\bowtie R_{m+1}.
\]
Thus, to complete the proof, it suffices to show that $T=R_{m+2}$.  
In fact, by the choice of the numbers $z_j$, we have $T\subseteq R_{m+2}$.  
By \cref{lem:deg:10}, the polynomial $W_{m+2}(x)$ has degree $m+2$. 
Thus, the zero-set~$R_{m+2}$ has cardinality at most $m+2$. 
Since it contains the subset $T$ with cardinality $m+2$, we conclude that $R_{m+2}=T$.
\end{proof}
\smallskip

Using \cref{lem:crt:itl:10} as the induction step, we now establish a criterion 
for a bound on the union $\cup_{n\ge1} R_n$ of the zero-sets, 
which will be used in the proof of \cref{thm:RR:10}.

\begin{cor}\label{cor:crt:itl:10}
Let $\{W_n(x)\}_{n\ge0}$ be a \oz, 
and let $R_n$ denote the zero-set of the polynomial $W_n(x)$. 
Suppose that there are numbers $\alpha$ and $\beta$,  with $\alpha<0<\beta$, 
such that \cref{cond:itl:alpha:10,cond:itl:beta:10} hold for all $m\ge 0$.
Then we have the three relations $|R_n|=n$, $R_n\subset(\alpha,\beta)$, 
and $R_{n}\bowtie R_{n-1}$, for all $n\ge 1$.
\end{cor}

\begin{proof} 
Since $W_1(x)=x$, we have $R_1=\{0\}$. Thus the inequality $\alpha<0<\beta$ 
is equivalent to the relation $R_1\subset (\alpha,\beta)$.
Since $W_0(x)=1$, we have $R_0=\emptyset$. 
As per \cref{def:itl:10}, the relation $R_1\bowtie R_0$ holds trivially. 
This proves the desired relations for $n=1$.

Now suppose that they are true for some index $n=m+1$, where $m\ge 0$. 
By induction, all the three relations in \eqref{prm:itl:10} hold true.
Since \cref{cond:itl:alpha:10,cond:itl:beta:10} hold by the premise, 
by \cref{lem:crt:itl:10}, we obtain the three relations in \eqref{rst:itl:10}.
In other words, the desired relations hold for the index $n=m+2$. This completes the proof.
\end{proof}

In order to establish the real-rootedness of every polynomial in a $(1,0)$-sequence,
we will construct two real numbers~$\alpha$ and~$\beta$ satisfying the premises of  \cref{cor:crt:itl:10}.
Inspired by \cref{lem:00}, we employ the following notations.

\begin{ntt}\label{ntt:10}
Let $\{W_n(x)\}_{n\ge0}$ be a \oz. We define
\begin{align}
A(x)&\=ax+b,\notag\\[4pt]
\label[def]{def:Delta:10} \Delta(x)&\=A^2(x)+4d,\\[4pt]
\label[def]{def:f:10} f(x)&\=2W_1(x)-A(x)\=(2-a)x-b,\\[4pt]
\label[def]{def:g+-:10} g^\pm(x)&\=\bigl(f(x)\pm\sqrt{\Delta(x)}\,\bigr)/2,\rmand 
\\[4pt]
\label[def]{def:g:10}
g(x)&\=g^-(x)g^+(x)\=(1-a)x^2-bx-d.
\end{align}
\end{ntt}

The roots $x_\Delta^\pm$ of the function $\Delta(x)$ are given in \cref{def:xd+-:10}, 
and the roots $x_g^\pm$ of the function $g(x)$ are given in  \cref{def:xg+-:10}. 
\Cref{lem:xd:10} collects some information for the numbers $x_\Delta^\pm$.

\begin{lem}\label{lem:xd:10}
Let $\{W_n(x)\}_{n\ge0}$ be a \oz. Let $n\ge 0$. Then we have
\begin{align}
\label{A:xd:10} A(x_\Delta^\pm) &\= ax_\Delta^\pm+b\=\pm2\sqrt{-d},\\[7pt]
\label{f:xd1:10} f(x_\Delta^-) &\= {2\over a}\Bigl((a-2)\sqrt{-d}-b\Bigr),\\[4pt]
\label{f:xd2:10} f(x_\Delta^+) &\= {2\over a}\Bigl((2-a)\sqrt{-d}-b\Bigr),  \rmand  \\[4pt]
\label{W:xd} W_n(x_\Delta^\pm)&\= 
\bgg{1+n\cdotp{f(x_\Delta^\pm)\over A(x_\Delta^\pm)}}\bgg{{A(x_\Delta^\pm)\over 2}}^n.
\end{align}
Consequently, we have the following.
\begin{align}
\label[ineq]{sgn:A:xd:10} &A(x_\Delta^-)\<0\<A(x_\Delta^+),\\[4pt]
\label[rl]{sgn:f:xd1:10} &f(x_\Delta^-)>0 \eqrl b<(a-2)\sqrt{-d}, \rmand \\[4pt]
\label[rl]{sgn:f:xd2:10} &f(x_\Delta^+)<0 \eqrl b>(2-a)\sqrt{-d}.
\end{align}
We also have the implication
\begin{equation}\label[rl]{imp:10}
f(x_\Delta^-)>0
\quad\Longrightarrow\quad
a>2
\quad\Longrightarrow\quad
f(x_\Delta^+)<0.
\end{equation}
\end{lem}

\begin{proof}
From \cref{ntt:10}, it is straightforward to compute \crefrange{A:xd:10}{f:xd2:10}.
Since $\Delta(x_\Delta^\pm)=0$, \Cref{lem:00} implies \cref{W:xd}.
\Cref{sgn:A:xd:10,sgn:f:xd1:10,sgn:f:xd2:10} follow from 
\cref{A:xd:10,,f:xd1:10,,f:xd2:10}, respectively.
From \cref{sgn:f:xd1:10}, we see that the inequality $f(x_\Delta^-)>0$ implies that $a>2$,
and consequently, $b\ge0>(2-a)\sqrt{-d}$. Then, from \cref{sgn:f:xd2:10}, we infer \cref{imp:10}.
\end{proof}

Our direction is toward showing in \cref{sec:xd:10} that 
the numbers $x_\Delta^\pm$ are limit points of the union $\cup_{n\ge1} R_n$.
As will be seen, \cref{imp:10} plays a crucial role in splitting cases in the statement of our main theorem, 
\cref{thm:10}. In \cref{lem:xg:10} and \cref{lem:J0:10} we are collecting some information 
regarding the numbers $x_g^\pm$ that we defined in the statement of \cref{thm:10:normalized}.
\smallskip

\begin{lem}\label{lem:xg:10}
Let $\{W_n(x)\}_{n\ge0}$ be a \oz, and suppose that the numbers $x_g^-$ and $x_g^+$ are both real.  
Then we have $W_n(x_g^\pm)=(x_g^\pm)^n$. 
Moreover, we have the following.
\begin{itemize}
\smallskip\item[(i)]
If $a\le 1$, then $0<x_g^+\le x_g^-$, and $W_n(x_g^\pm)>0$.
\smallskip\item[(ii)]
If $a>1$, then $x_g^-<0<x_g^+$, $W_n(x_g^-)(-1)^n>0$, and $W_n(x_g^+)>0$.
\end{itemize}
\end{lem}

\begin{proof}
See \cref{apd:xg:10}.
\end{proof}
\smallskip

We will show in \cref{sec:xg:10} that, in some cases, the numbers $x_g^\pm$ are also limit points 
of the union $\cup_{n\ge1} R_n$.  In order to give a bound for the union $\cup_{n\ge1} R_n$, we need to be clear about the ordering among the numbers  $x_\Delta^\pm$ and $x_g^\pm$.  \Cref{lem:J0:10} collects necessary information for this purpose.
\smallskip

\begin{lem}\label{lem:J0:10}
Suppose that $f(x_\Delta^+)<0$, where the function $f(x)$ is given by \cref{def:f:10}.
Then the numbers~$x_g^\pm$ are well-defined and real. 
Moreover, we have $x_g^+>x_\Delta^+$.  If additionally we have $f(x_\Delta^-)>0$, then $x_g^-<x_\Delta^-$.
\end{lem}
\begin{proof}
See \cref{apd:J0:10}.
\end{proof}
\smallskip

We are now ready to define a suitable interval $(\alpha,\beta)$ to be used in applying \cref{cor:crt:itl:10}.
\smallskip

\begin{ntt}\label{ntt:J0:10}
Let $J_\Delta=(x_\Delta^-,\,x_\Delta^+)$.
With the aid of \cref{imp:10}, we can define the interval 
\begin{equation}\label[def]{def:J0:10}
J_0=\begin{cases}
(x_\Delta^-,\,x_\Delta^+),&\text{if $f(x_\Delta^+)\ge0$};\\[5pt]
(x_g^-,\,x_g^+),&\text{if $f(x_\Delta^-)>0$};\\[5pt]
(x_\Delta^-,\,x_g^+),
&\text{otherwise}.
\end{cases}
\end{equation}
Note that $x_\Delta^-<x_\Delta^+$.
We see from \cref{lem:J0:10} that $J_0$ is a well-defined interval, 
and that it contains $J_\Delta$ as a non-empty subinterval, namely,
\begin{equation}\label[rl]{JJ0:10}
J_\Delta\subseteq J_0.
\end{equation}
\end{ntt}

This has prepared us to establish the distinct-real-rootedness of the polynomials $W_n(x)$.

\begin{thm}\label{thm:RR:10}
Let $\{W_n(x)\}_{n\ge0}$ be a \oz. Then each of the polynomials  $W_n(x)$ is distinct-real-rooted.  
Moreover, we have $R_n\subset J_0$ and $R_{n+1}\bowtie R_{n}$.
\end{thm}

\begin{proof}
Let $(\alpha,\beta)$ denote the interval $J_0$.  By \cref{def:J0:10}, we have
\[
\alpha=\begin{cases}
x_g^-,&\text{if $f(x_\Delta^-)>0$},\\[5pt]
x_\Delta^-,&\text{if $f(x_\Delta^-)\le 0$};
\end{cases}
\rmand
\beta=\begin{cases}
x_g^+,&\text{if $f(x_\Delta^+)<0$},\\[5pt]
x_\Delta^+,&\text{if $f(x_\Delta^+)\ge0$}.
\end{cases}
\]
To apply \cref{cor:crt:itl:10}, 
we will first show that $\alpha<0$ and $\beta>0$ 
and then that \cref{cond:itl:alpha:10,cond:itl:beta:10} hold for all $m\ge 0$.

First, we show that $\alpha<0$. 
If $f(x_\Delta^-)>0$,  then we have $a>2$ from \cref{imp:10}, 
which implies that $x_g^-<0$, by \cref{lem:xg:10}~(ii). 
Alternatively, if $f(x_\Delta^-)\le0$, 
then it follows from \cref{def:xd+-:10} of the number~$x_\Delta^-$ 
as $(-b-2\sqrt{-d})/a$ that $x_\Delta^-<0$, since $a>0$ and $b\ge 0$.  This proves $\alpha<0$.

Second, we show that $\beta>0$.
If $f(x_\Delta^+)<0$,  then we have $x_g^+>0$, by \cref{lem:xg:10}. 
Alternatively, if $f(x_\Delta^+)\ge 0$, 
we deduce from \cref{sgn:f:xd2:10} that $b\le(2-a)\sqrt{-d}<2\sqrt{-d}$. 
Thus, from \cref{def:xd+-:10} of the number~$x_\Delta^+$ as $(-b+2\sqrt{-d})/a$, 
if follows that $x_\Delta^+>0$.  This proves $\beta>0$.

Now, for \cref{cond:itl:alpha:10,cond:itl:beta:10}, we need to show that
\begin{equation}\label[ineq]{pf15}
W_m(x)W_{m+2}(x)>0,\qquad\text{for all $m\ge0$, and for both $x\in\{\alpha,\beta\}$}.
\end{equation}
Since $\{\alpha,\beta\}\subseteq\{x_g^\pm,\,x_\Delta^\pm\}$, we have $x\in\{x_g^\pm,\,x_\Delta^\pm\}$.

\smallskip

When $x\in\{x_g^\pm\}$,  we have from \cref{lem:xg:10} that $W_m(x)W_{m+2}(x)=x^{2m+2}\ge0$.  
\Cref{lem:xg:10} also tells us $x_g^\pm\ne0$. Thus we have $W_m(x_g^\pm) W_{m+2}(x_g^\pm)>0$.

\smallskip

When $x\in\{x_\Delta^\pm\}$, \cref{W:xd} implies that
\[
W_m(x)W_{m+2}(x)
\=\bgg{1+m\cdotp{f(x)\over A(x)}}
\bgg{1+(m+2)\cdotp{f(x)\over A(x)}}
\bgg{{A(x)\over 2}}^{2m+2}.
\]
To show \cref{pf15}, it suffices (since $m\ge0$) to show that $f(x)/A(x)\ge0$. 
For $x=x_\Delta^-$, we have $x=\alpha$, with the additional condition $f(x_\Delta^-)\le 0$.  
From \cref{sgn:A:xd:10}, we infer that $f(x_\Delta^-)/A(x_\Delta^-)\ge0$.  
If $x=x_\Delta^+$, we have $x=\beta$, with the additional condition $f(x_\Delta^+)\ge 0$.  
Again from \cref{sgn:A:xd:10}, we infer that $f(x_\Delta^+)/A(x_\Delta^+)\ge0$.  
\end{proof}

We see from \cref{thm:RR:10} that the union $\cup_{n\ge1}R_n$ is contained within the interval~$J_0$.  
We will show that the interval~$J_0$ is sharp as a bound of that union,
by establishing in the next three subsections  that each of its endpoints is a limit point of that union; 
see \cref{thm:xg1:10}, \cref{thm:xg2:10}, \cref{thm:xd2:10}, and \cref{thm:xd1:10}.
Before this, we highlight the important role of the subinterval~$J_\Delta$,
by proving that it contains almost all members of the union $\cup_{n\ge1}R_n$.

\medskip\medskip
\subsection{The quasi-bound $J_\Delta$}  

By \cref{thm:RR:10}, each of the polynomials $W_n(x)$ is distinct-real-rooted.  
In order to describe the position of the smallest and the largest roots clearly, 
we introduce the following notation.
\smallskip

\begin{dfn}
A set $\{x_1,x_2,\ldots,x_n\}$ of real numbers is said to be {\em ordered} or {\em in increasing order}, if $x_1< x_2< \cdots <x_n$.
\end{dfn}
\smallskip

The following theorem shows that the subinterval $J_\Delta$ contains 
all except at most two roots of each polynomial $W_n(x)$.
\smallskip

\begin{thm}\label{thm:bound:10}
Let $\{W_n(x)\}_{n\ge0}$ be a \oz, in which each polynomial $W_n(x)$ 
has the ordered zero-set $R_n=\{\xi_{n,1},\,\xi_{n,2},\,\ldots,\,\xi_{n,n}\}$.

\smallskip\noindent(i) When $f(x_\Delta^-)>0$, we have the following.
\begin{itemize}
\medskip\item
If $n<-{A(x_\Delta^-)/f(x_\Delta^-)}$, then $\xi_{n,1}>x_\Delta^-$.
\medskip\item
If $n=-{A(x_\Delta^-)/f(x_\Delta^-)}$, then $\xi_{n,1}=x_\Delta^-$.
\medskip\item
If $n>-{A(x_\Delta^-)/f(x_\Delta^-)}$, then $\xi_{n,1}<x_\Delta^-$. In this case, we have $x_\Delta^-<\xi_{n,2}$ for all $n\ge 2$.
\end{itemize}

\smallskip\noindent (ii) When $f(x_\Delta^+)<0$, we have the following.
\begin{itemize}
\medskip\item
If $n<-{A(x_\Delta^+)/f(x_\Delta^+)}$, then $\xi_{n,n}<x_\Delta^+$.
\medskip\item
If $n=-{A(x_\Delta^+)/f(x_\Delta^+)}$, then $\xi_{n,n}=x_\Delta^+$.
\medskip\item
If $n>-{A(x_\Delta^+)/f(x_\Delta^+)}$, then $\xi_{n,n}>x_\Delta^+$. 
In this case, we have $\xi_{n,\,n-1}<x_\Delta^+$ for all $n\ge 2$.
\end{itemize}
\end{thm}

\begin{proof}
We define
\begin{equation}\label[def]{def:n+-:10}
n^\pm\=-\frac{A(x_\Delta^\pm)}{f(x_\Delta^\pm)}.
\end{equation}
Then \cref{W:xd} can be rewritten as 
\begin{equation}\label{pf:Wxd:10}
W_n(x_\Delta^\pm)\= 
\bgg{1-\frac{n}{n^\pm}}\bgg{{A(x_\Delta^\pm)\over 2}}^n.
\end{equation} 
We shall show (i) and (ii) individually.

\medskip\noindent{\bf (i)} Suppose that 
\begin{equation}\label[ineq]{pf:fD>0:10}
f(x_\Delta^-)>0. 
\end{equation}
From \cref{sgn:A:xd:10}, we see that 
\begin{equation}\label[ineq]{pf:Axd1<0:10}
A(x_\Delta^-)<0.
\end{equation}
Together with \cref{pf:fD>0:10}, we can see from \cref{def:n+-:10} that 
\begin{equation}\label[ineq]{pf:n->0:10}
n^->0.
\end{equation}
Now, in view of \cref{pf:Wxd:10,pf:Axd1<0:10,pf:n->0:10}, we infer that  
\begin{equation}\label[ineq]{sgn:Wxd1:10}
(-1)^n(n-n^-)W_n(x_\Delta^-)
\=-\frac{(n-n^-)^2}{n^-}\bgg{{-A(x_\Delta^-)\over 2}}^n
\;\le\;0,
\end{equation}
where the equality holds if and only if $n=n^-$.
On the other hand, \cref{imp:10} with \cref{pf:fD>0:10} tells us $a>2$. Thus, we infer from \cref{lem:xg:10}~(ii) that
\begin{equation}\label[ineq]{sgn:Wxg1:10}
W_n(x_g^-)(-1)^n\>0.
\end{equation}
Multiplying \cref{sgn:Wxd1:10,sgn:Wxg1:10}, we find 
\begin{equation}\label[ineq]{sgn:W:xg1xd1:10}
W_n(x_g^-)W_n(x_\Delta^-)(n-n^-)\;\le\;0,
\end{equation}
where the equality holds if and only if $n=n^-$.

According to \cref{def:J0:10}, we have $J_0=(x_g^-,x_g^+)$.
From \cref{JJ0:10}, we infer that 
\begin{equation}\label[ineq]{ord:xgxdxdxg:10}
x_g^-<x_\Delta^-<x_\Delta^+<x_g^+.
\end{equation}
Thus we can define the interval 
\[
J^-=(x_g^-,x_\Delta^-).
\] 
By \cref{thm:RR:10}, the roots of the polynomials $W_{n+1}(x)$ and $W_n(x)$ are interlacing in the interval $J_0$.
Since the interval $J^-$ is ``a left part'' of the interval $J_0$,
we infer that in $J^-$, the polynomial $W_{n+1}(x)$ has the same number of roots or one more root than the polynomial $W_n(x)$.
In notation, we let
\begin{equation}\label[def]{def:N:10}
N_n\=|R_n\cap J^-|.
\end{equation}
Then the above argument can be expressed as 
\begin{equation}\label[rl]{diff:N-:10}
N_{n}-N_{n-1}\in\{0,1\}\all{n\ge1}.
\end{equation}
Below we show Result (i) by bootstrapping.

\medskip\noindent{\bf Case $n<n^-$.} 
Since the zero-set $R_0$ is empty, we infer from \cref{def:N:10} that $N_0=0$.
We claim that 
\[
N_n=0,\all{n<n^-}.
\]

If $n^-\le 1$, then the claim holds true trivially. Below we suppose that $n^->1$.
We proceed by induction on $n$. Suppose that $N_{m-1}=0$ for some $1\le m<n^-$.
From \cref{diff:N-:10}, we infer that 
\begin{equation}\label{pf1:Nn<:10}
N_{m}\in\{0,1\}.
\end{equation}

Since $n<n^-$, \cref{sgn:W:xg1xd1:10} implies that $W_n(x_g^-)W_n(x_\Delta^-)>0$.
In other words, at the ends of the interval $J^-$, 
the continuous function $W_n(x)$ admits the same sign.
Therefore, from the intermediate value theorem, 
we infer that $W_n(x)$ has an even number of roots in the interval $J^-$, namely, 
the integer~$N_n$ is even as if $n<n^-$.
In view of \cref{pf1:Nn<:10}, we find that $N_m=0$.
This proves the claim.

The claim, in fact, states that the polynomial $W_n(x)$ has no roots in the interval $J^-=(x_g^-,x_\Delta^-)$.
Recall from \cref{thm:RR:10} that all roots of $W_n(x)$ lie in the interval $J_0=(x_g^-,\,x_g^+)$.
Therefore, the bound of the roots of $W_n(x)$ can be improved to the interval $[x_\Delta^-,x_g^+)$.
Moreover, it is easy to see that $W_n(x_\Delta^-)\neq0$ from \cref{pf:Wxd:10}. Hence, we obtain that $R_n\subset(x_\Delta^-,x_g^+)$.
In particular, we have
\begin{equation}\label[rst]{pf:n<m-:10}
\xi_{n,1}\>x_\Delta^-\all{n<n^-}.
\end{equation}

\medskip\noindent{\bf Case $n=n^-$.}
From \cref{pf:Wxd:10}, we see that 
\[
W_n(x_\Delta^-)\=\bgg{1-\frac{n}{n^-}}\bgg{{A(x_\Delta^-)\over 2}}^n=0.
\]
In other words, the number $x_\Delta^-$ is a root of the polynomial $W_{n}(x)$.
We shall show that $x_\Delta^-$ is the smallest root of $W_n(x)$.

If $n=1$, then the polynomial $W_{n}(x)=W_1(x)=x$ has only one root.
Therefore, the root $x_\Delta^-$ has to be the smallest root.
When $n\ge 2$, the interlacing property $R_{n}\bowtie R_{n-1}$ implies that $\xi_{n,2}\>\xi_{n-1,\,1}$.
From \cref{pf:n<m-:10}, we infer that $\xi_{n-1,\,1}\>x_\Delta^-$. 
It follows that $\xi_{n,2}>x_\Delta^-$. In other words,
the second smallest root of the polynomial $W_{n}(x)$ is larger than the root $x_\Delta^-$.
It turns out immediately that the root $x_\Delta^-$ is the smallest one, namely,
\begin{equation}\label[rst]{pf:n=m-:10}
\xi_{n,1}\=x_\Delta^-,\qquad\text{if $n=n^-$}.
\end{equation}

\medskip\noindent{\bf Case $n>n^-$.}
In this case, \cref{sgn:W:xg1xd1:10} reads that $W_n(x_g^-)W_n(x_\Delta^-)<0$.
As in the case $n<n^-$, we can deduce by using the intermediate value theorem that 
the integer $N_n$ is odd.
In particular, we have $N_n\ge1$. 

We claim that 
\begin{equation}\label[rst]{pf:Nn=1:10}
N_n\=1\all{n>n^-}.
\end{equation}
Combining \cref{pf:n<m-:10,pf:n=m-:10}, we see that for any $m\le n^-$,
the smallest root of the polynomial $W_m(x)$ is at least $x_\Delta^-$. It follows that $N_m=0$.
In particular, we have $N_{\fl{n^-}}=0$.
Thus, by \cref{diff:N-:10}, we have 
\[
N_{\fl{n^-}+1}\in\{0,1\}.
\] 
Since $\fl{n^-}+1>n^-$, the integer~$N_{\fl{n^-}+1}$ is odd.
Thus $N_{\fl{n^-}+1}=1$. 
Proceeding by induction, we can suppose that $N_{n}=1$ for some $n>n^-$.
By \cref{diff:N-:10}, we infer that $N_{n+1}\in\{1,2\}$. Since the integer~$N_{n+1}$ is odd,
we deduce that $N_{n+1}=1$. This confirms the claim.

By \cref{def:N:10} of the number~$N_n$ as $|R_n\cap(x_g^-,x_\Delta^-)|$, 
\cref{pf:Nn=1:10} implies that 
$\xi_{n,1}<x_\Delta^-$ if $n>n^-$.
From \cref{pf:Wxd:10}, we see that $W_n(x_\Delta^-)\neq0$.
Together with \cref{pf:Nn=1:10}, we obtain that
$\xi_{n,2}>x_\Delta^-$ as if $n>n^-$ and $n\ge 2$.

This completes the proof of (i).

\medskip\noindent{\bf (ii)}
Suppose that $f(x_\Delta^+)<0$. 
Analogous to \cref{sgn:W:xg1xd1:10}, one may show that 
\begin{equation}\label[ineq]{sgn:W:xg2xd2:10}
W_n(x_\Delta^+)W_n(x_g^+)(n-n^+)\le 0,
\end{equation}
where the equality holds if and only if $n=n^+$.

According to \cref{def:J0:10}, we have $J_0=(x_g^-,x_g^+)$ or $J_0=(x_\Delta^-,x_g^+)$.
From \cref{JJ0:10}, we can define 
\[
N_n^+\=|R_n\cap (x_\Delta^+,x_g^+)|.
\]
Since the interval $(x_\Delta^+,x_g^+)$ is ``a right part'' of the interval $J_0$,
the interlacing property implies that
\begin{equation}\label[rl]{diff:N+:10}
N_n^+-N_{n-1}^+\in\{0,1\}\all{n\ge1}.
\end{equation}
Similar to (i), by using the intermediate value theorem, the parity arguments,
with the aids of \cref{sgn:W:xg2xd2:10,diff:N+:10},
we can show the desired results in (ii).

This completes the proof of \cref{thm:bound:10}.
\end{proof}
\medskip

\subsection{The numbers $x_g^\pm$ can be limit points}\label{sec:xg:10}

From \cref{thm:bound:10}, we see that when $f(x_\Delta^-)>0$, 
the smallest roots are eventually less than the number~$x_\Delta^-$.  
\Cref{thm:xg1:10} shows that they converge to the number $x_g^-$.  
Similarly, when $f(x_\Delta^+)<0$, the largest roots are eventually larger than the number~$x_\Delta^+$.  
\Cref{thm:xg2:10} shows that they converge to the number $x_g^+$ in that case.
\smallskip

In the proofs of these two theorems, we will use \cref{sol:W:d>0}, 
which is obtained from \cref{lem:00} straightforwardly.
\begin{prop}
If $\Delta(x_0)>0$, then  
\begin{equation}\label[fml]{sol:W:d>0}
 W_n(x_0)
\={\bigl(A(x_0)\!+\!\sqrt{\Delta(x_0)}\,\bigr)^{\!n}\over 2^{n}\sqrt{\Delta(x_0)}}
\left[g^+\!(x_0)-g^-\!(x_0)\bgg{{A(x_0)-\sqrt{\Delta(x_0)}\over A(x_0)+\sqrt{\Delta(x_0)}}\,}^{\!\!n}\right]. 
\end{equation}
\end{prop}

\begin{thm}\label{thm:xg1:10}
Let $\{W_n(x)\}_{n\ge0}$ be a \oz. 
If $f(x_\Delta^-)>0$, then $\xi_{n,1}\searrow x_g^-$.
\end{thm}

\begin{proof}
Suppose that $f(x_\Delta^-)>0$. Then $a>2$ by \cref{imp:10}.

By \cref{thm:bound:10}~(i), we have $\xi_{n,1}\in(x_g^-,\,x_\Delta^-)$ for large~$n$.
The property $R_{n+1}\bowtie R_{n}$ 
obtained in \cref{thm:RR:10} implies that the sequence $\xi_{n,1}$ decreases.
Therefore, there exists a number $\xi^-\in[x_g^-,\,x_\Delta^-)$ such that 
$\lim_{n\to\infty}\xi_{n,1}=\xi^-$.
Suppose, to the contrary, that $\xi^-\ne x_g^-$. 
Let $x_0\in(x_g^-,\,\xi^-)$. Then 
\begin{equation}\label[rl]{pf:order2:10}
x_g^-\<x_0\<\xi^-\<x_\Delta^-\<-b/a.
\end{equation}
We claim that
\begin{equation}\label[ineq]{pf:claim1:10}
g^-(x_0)\;\le\;0.
\end{equation}

Since the interval $(x_g^-,\,\xi^-)$ contains no roots of any polynomial $W_n(x)$,
by the intermediate value theorem, we infer that 
\[
W_n(x_g^-)W_n(x_0)\>0.
\]
Since $a>2$, by \cref{lem:xg:10}~(ii), we have $W_n(x_g^-)(-1)^n>0$. Together with the above inequality, we derive that
$W_n(x_0)(-1)^n>0$. In particular, we have 
\begin{equation}\label[ineq]{pf:x0:xg1:10}
W_n(x_0)\>0\qquad\text{for all positive even integers $n$}.
\end{equation}
On the other hand, since $a>0$, the function $A(x)=ax+b$ is increasing.
From \cref{pf:order2:10}, we see that $x_0<-b/a$, and thus, $A(x_0)<A(-b/a)=0$. 
Since $d<0$, from \cref{def:Delta:10}, we deduce that
\[
A(x_0)-\sqrt{\Delta(x_0)}\<A(x_0)+\sqrt{\Delta(x_0)}\=A(x_0)+\sqrt{A^2(x_0)+4d}\<0.
\]
Therefore, we have
\begin{equation}\label[ineq]{pf:AA:xg1:10}
{A(x_0)-\sqrt{\Delta(x_0)}\over A(x_0)+\sqrt{\Delta(x_0)}}\>1.
\end{equation}
Set $n$ to be a large even integer. 
By \cref{def:xd+-:10}, the function~$\Delta(x)=(ax+b)^2+4d$ is a quadratic polynomial with a positive leading coefficient.
From \cref{pf:order2:10}, we see that $x_0<x_\Delta^-$, and thus, $\Delta(x_0)>0$.
By \cref{pf:x0:xg1:10}, the right hand side of \cref{sol:W:d>0} is positive.
In view of \cref{pf:AA:xg1:10},
we deduce that $g^-(x_0)\leq0$. This confirms the claim.

Since $a>2$, by \cref{def:f:10}, the function $f(x)=(2-a)x-b$ is decreasing. 
Since $x_0<x_\Delta^-$, we infer that $f(x_0)>f(x_\Delta^-)>0$.
Thus, by \cref{def:g+-:10}, we obtain that 
$g^+(x_0)=(f(x_0)+\sqrt{\Delta(x_0)})/2>0$.
Together with \cref{pf:claim1:10}, we derive that 
\begin{equation}\label[ineq]{pf23}
g(x_0)\=g^-(x_0)g^+(x_0)\;\le\;0.
\end{equation}

From \cref{def:g:10}, we see that $g(x)=(1-a)x^2-bx-d$
is a quadratic polynomial with negative leading coefficient.
Thus we have $g(x)>0$ for all $x\in(x_g^-,\,x_g^+)$.
By \cref{lem:J0:10}, we have $x_g^-<x_\Delta^-<x_\Delta^+<x_g^+$.
Together with \cref{pf:order2:10}, we deduce that $x_0\in(x_g^-,\,x_g^+)$, and thus, 
$g(x_0)>0$, contradicting \cref{pf23}. This completes the proof.
\end{proof}

To show that the number $x_g^+$ can be a limit point, we need the following lemma.
\begin{lem}\label{lem:f:xg2:10}
If $f(x_\Delta^+)<0$, then $f(x)<0$ for all $x\in(x_\Delta^+,\,x_g^+)$.
\end{lem}

\begin{proof}
See \Cref{apd:f:10}.
\end{proof}

The proof of the limit point $x_g^+$ is similar to the proof of the limit point $x_g^-$.

\begin{thm}\label{thm:xg2:10}
Let $\{W_n(x)\}_{n\ge0}$ be a \oz. 
If $f(x_\Delta^+)<0$, then $\xi_{n,n}\nearrow x_g^+$.
\end{thm}

\begin{proof}
Suppose that $f(x_\Delta^+)<0$.
Suppose, by way of contradiction, that the convergent point of the largest roots is not the number $x_g^+$. 
Then, there exists a number $x_0\in (x_\Delta^+,\,x_g^+)$ such that $W_n(x_0)W_n(x_g^+)>0$, that is,
\begin{equation}\label[ineq]{pf:x0:xg2:10}
W_n(x_0)>0.
\end{equation}
Similar to the proof of \cref{thm:xg1:10}, we can show that $A(x_0)>0$, $\Delta(x_0)>0$, and 
\[
A(x_0)+\sqrt{\Delta(x_0)}\>A(x_0)-\sqrt{\Delta(x_0)}\=A(x_0)-\sqrt{A^2(x_0)+4d}\>0.
\]
Therefore, we have
\begin{equation}\label[rl]{pf:AA:xg2:10}
{A(x_0)-\sqrt{\Delta(x_0)}\over A(x_0)+\sqrt{\Delta(x_0)}}\;\in\;(0,1).
\end{equation}
Set $n$ to be a large even integer. By \cref{pf:x0:xg2:10}, the right hand side of \cref{sol:W:d>0} is positive.
In view of \cref{pf:AA:xg2:10}, we infer that 
\begin{equation}\label{pf30}
g^+(x_0)\geq0.
\end{equation}
By \cref{lem:f:xg2:10}, we have $f(x_0)<0$.
Therefore, $g^-(x_0)=f(x_0)-\sqrt{\Delta(x_0)}<0$, and thus, 
\[
g(x_0)=g^-(x_0)g^+(x_0)\le0.
\]
Below, we show that $g(x_0)>0$, which implies an immediate contradiction.
Recall from \cref{def:g:10} that $g(x)=(1-a)x^2-bx-d$.
\begin{itemize}
\smallskip\item
When $a<1$, the quadratic function $g(x)$ has a positive leading coefficient.
On the other hand, by \cref{lem:xg:10}~(i), 
the number $x_g^+$ is the smaller root of the function~$g(x)$.
Therefore, we have $g(x)>0$ for all $x<x_g^+$. In particular, we have $g(x_0)>0$.
\smallskip\item
When $a=1$, we have $g(x)=-bx-d$. 
If $b=0$, then $g(x)=-d>0$ for all $x\in\R$. Otherwise $b>0$, then the function $g(x)$ is decreasing.
Thus, for any $x<x_g^+$, we have $g(x)>g(x_g^+)=0$. In particular, we have $g(x_0)>0$.
\smallskip\item
When $a>1$, the leading coefficient of the quadratic function $g(x)$ is negative.
On the other hand, by \cref{lem:xg:10}~(ii),
the number $x_g^+$ is the larger root of the function~$g(x)$.
Therefore, we have $g(x)>0$ for all $x\in(x_g^-,\,x_g^+)$.
In particular, we have $g(x_0)>0$, by \cref{lem:J0:10}.
\end{itemize}
This completes the proof.
\end{proof}

\medskip
\subsection{The numbers $x_\Delta^\pm$ are limit points}\label{sec:xd:10}

In this subsection, we will show that the numbers $x_\Delta^\pm$ are limit points of the union $\cup_{n\ge1}R_n$. 
For proof convenience, we will adopt the polar coordinate system.

\begin{dfn}\label{def:pv}
Let $\theta\in\R$. We define the {\em principal value} of the number~$\theta$, denote by~$\pv(\theta)$, 
to be the unique number $\theta'\in[\,0,2\pi)$ such that the difference $\theta-\theta'$ is an integral multiple of the number~$2\pi$.
In the polar coordinate system, we adopt the wording
\begin{itemize}
\smallskip\item
{\em the angle $\theta$}, to mean the angle of size $\theta$;
\smallskip\item
{\em the ray $\theta$}, to mean the ray starting from the origin with the incline angle $\pv(\theta)$; and
\smallskip\item
{\em the line $\theta$}, to mean the line on which lies the ray $\theta$.
\end{itemize} 
Let $\psi\in[\,0,\,\pi)$. We say that the angle~$\theta$ lies 
\begin{itemize}
\smallskip\item
{\em to the left of the line~$\psi$}, if $\theta\in(\psi+2k\pi,\,\psi+(2k+1)\pi)$ for some integer $k$;
\smallskip\item
{\em to the right of the line $\psi$}, if $\theta\in(\psi+(2k-1)\pi,\,\psi+2k\pi)$ for some integer $k$;
\smallskip\item
{\em on the line~$\psi$}, if the line $\theta$ coincides with the line $\psi$.
\end{itemize}
\end{dfn}
For example, when $\psi\in[\,0,\pi/2)$ (resp., $\psi\in(\pi/2,\pi)$), 
the angle~$\theta$ lies to the left of the line $\psi$ if and only if 
it is above (resp., below) the line $\psi$, intuitively.


To characterize the sign of the value $W_n(x)$ for the real numbers $x$ such that $\Delta(x)<0$, we need the following notation.

\begin{ntt}\label{ntt:theta:psi:10}
Let $\{W_n(x)\}_{n\ge0}$ be a \oz. 
Let $x\in\R$ such that $\Delta(x)<0$.
Define the angle $\theta_{x}\in(0,\pi)$ by 
\begin{equation}\label[def]{def:theta0:10}
\tan\theta_{x}\={\sqrt{-\Delta(x)}\over A(x)}.
\end{equation}
Define the angle $\psi_x\in(0,\pi)$ by
\begin{equation}\label[def]{def:l0:10}
\tan\psi_x\={-\sqrt{-\Delta(x)}\over (2-a)x-b}.
\end{equation}
\end{ntt}

Here is a characterization for the sign of the value~$W_n(x_0)$.

\begin{thm}\label{thm:sgn:d<0:10}
Let $\{W_n(x)\}_{n\ge0}$ be a \oz. Let $x\in\R$ such that $\Delta(x)<0$.
Then we have
\begin{itemize}
\smallskip\item
$W_n(x)<0$ if and only if the angle $n\theta_{x}$ lies to the left of the line $\psi_{x}$;
\smallskip\item
$W_n(x)=0$ if and only if  the angle $n\theta_{x}$ lies on the line $\psi_{x}$;
\smallskip\item
$W_n(x)>0$ if and only if  the angle $n\theta_{x}$ lies to the right of the line $\psi_{x}$.
\end{itemize}
\end{thm}

\begin{proof}
Since $\Delta(x)<0$, the complex number $A(x)+\sqrt{\Delta(x)}$ has the real part $A(x)$
and the imaginary part $\sqrt{-\Delta(x)}$. Therefore, by \cref{def:theta0:10} of the angle $\theta_{x}$,
we have
\[
A(x)+\sqrt{\Delta(x)}=Re^{i\theta_{x}},
\]
where $R=\sqrt{A^2(x)-\Delta(x)}$.
Let 
\[
h=\bigl((2-a)x-b\bigr)\big/\sqrt{-\Delta(x)},
\rmand
F=\cos(n\theta_{x})+h\cdotp\sin(n\theta_{x}).
\]
Since $\Delta(x)<0$, by \cref{lem:00}, we have 
\[
W_n(x)
\=\bgg{R\over2}^n\bgg{\cos\bigl(n\theta_{x}\bigr)+{(2-a)x-b\over \sqrt{-\Delta(x)}}\cdotp\sin\bigl(n\theta_{x}\bigr)}
\=\bgg{R\over2}^n F.
\]
Since $R>0$, we have $W_n(x)>0$ if and only if $F>0$.
On the other hand, in view of \cref{def:l0:10}, the line $\psi_x$ has slope $-1/h$.

If $h=0$, we have $\psi_x=\pi/2$.
In this case, we have $F=\cos(n\theta_{x})$, and thus, the above sign relation reduces to that $W_n(x)$ and $\cos\psi_x$
have the same sign. 
In other words, we have $W_n(x)>0$ if and only if
the angle $n\theta_{x}$ lies in the open right half-plane; and
$W_n(x)<0$ if and only if the angle $n\theta_{x}$ lies in the left open half-plane.
Consequently, we have $W_n(x)=0$ if and only if the line $n\theta_{x}$ coincides with the vertical line $\pi/2$.
This proves the desired relations.
Below we can suppose that $h\ne 0$.

Assume that $h>0$.
From the definition of the function~$F$, it is elementary to deduce the following equivalence relation
\begin{align*}
F>0
&\eqrl
\sin(n\theta_{x})>-{1\over h}\cos(n\theta_{x})\\
&\eqrl
\begin{cases}
\tan(n\theta_{x})>-1/h,&\text{if $\cos(n\theta_{x})>0$};\\[5pt]
\sin(n\theta_{x})>0,&\text{if $\cos(n\theta_{x})=0$};\\[5pt]
\tan(n\theta_{x})<-1/h,&\text{if $\cos(n\theta_{x})<0$}.
\end{cases}
\end{align*}
Therefore, we have $W_n(x)>0$ if and only if the angle $n\theta_{x}$ belongs to the set $S_1\cup S_2\cup S_3$, where
\begin{align*}
S_1&
\=\{\psi\in\R\colon\,\cos\psi>0\text{ and }\tan\psi>-1/h\}
\=\cup_{k\in\Z}(\psi_x-\pi+2k\pi,\,\pi/2+2k\pi),\\[5pt]
S_2&
\=\{\psi\in\R\colon\,\cos\psi=0\text{ and }\sin\psi>0\}
\=\cup_{k\in\Z}\{\pi/2+2k\pi\},\\[5pt]
S_3&
\=\{\psi\in\R\colon\,\cos\psi<0\text{ and }\tan\psi<-1/h\}
\=\cup_{k\in\Z}(\pi/2+2k\pi,\,\psi_x+2k\pi).
\end{align*}
It is routine to deduce their union, which is
\[
S_1\cup S_2\cup S_3
=\cup_{k\in\Z}(\psi_x+(2k-1)\pi,\,\psi_x+2k\pi).
\]
From \cref{def:pv}, we see that $W_n(x)>0$ if and only if the angle $n\theta_x$ lies to the right of the line~$\psi_x$.
By symmetry, we have $W_n(x)<0$ if and only if the angle~$n\theta_x$ lies to the left of the line~$\psi_x$. 
It follows that $W_n(x)=0$ if and only if the angle~$n\theta_x$ lies on the line~$\psi_x$.

When $h<0$, we have the following equivalence relation in the same vein:
\[
F>0
\eqrl
\begin{cases}
\tan(n\theta_{x})<-1/h,&\text{if $\cos(n\theta_{x})>0$};\\[5pt]
\sin(n\theta_{x})<0,&\text{if $\cos(n\theta_{x})=0$};\\[5pt]
\tan(n\theta_{x})>-1/h,&\text{if $\cos(n\theta_{x})<0$}.
\end{cases}
\]
In the same way we can find the same desired relations. This completes the proof.
\end{proof}

Now we are ready to show that the number $x_\Delta^+$ is a limit point.

\begin{thm}\label{thm:xd2:10}
Let $\{W_n(x)\}_{n\ge0}$ be a \oz, with ordered zero-set 
\[
\{\xi_{n,1},\,\xi_{n,2},\,\ldots,\,\xi_{n,n}\}.
\]  
Then we have 
$\xi_{n,\,n+1-i}\nearrow x_\Delta^+$ as $n\to\infty$,
for all $i\ge 2$ if $f(x_\Delta^+)<0$, and for all $i\ge1$ otherwise.
\end{thm}

\begin{proof}
Let $i\ge 1$. From the property $R_{n+1}\bowtie R_n$ obtained in \cref{thm:RR:10},
we see that the sequence~$\xi_{n,\,n+1-i}$ ($n\ge i$) increases as $n\to\infty$. 
Since all the roots are bounded by the interval $J_0$, the sequence $\xi_{n,\,n+1-i}$ converges.
Suppose that $\lim_{n\to\infty}\xi_{n,\,n+1-i}=\ell_i$.

Suppose that $f(x_\Delta^+)<0$ and $i\ge2$.
From \cref{thm:bound:10}~(ii), we see that $\xi_{n,\,n+1-i}<x_\Delta^+$ for large~$n$,
which implies that $\ell_i\le x_\Delta^+$. Suppose, to the contrary, that $\ell_i<x_\Delta^+$. 
When $n$ is large, the polynomial $W_n(x)$ has exactly $i-2$ distinct roots in the interval $(\ell_i,\,x_\Delta^+)$,
that is, the roots $\xi_{n,\,n+2-i}$, $\xi_{n,\,n+3-i}$, $\ldots$, $\xi_{n,\,n-1}$.
Thus, by the intermediate value theorem, we infer that 
\[
W_n(\ell_i)W_n(x_\Delta^+)(-1)^{i-2}>0\qquad\text{for large $n$}.
\]
On the other hand, we have $A(x_\Delta^+)>0$ by \cref{sgn:A:xd:10}.
Since $f(x_\Delta^+)<0$, by \cref{W:xd}, we infer that $W_n(x_\Delta^+)<0$ for large~$n$.
Multiplying it with the above inequality, we obtain that
\begin{equation}\label{sgn:W:ell:10}
W_n(\ell_i)(-1)^{i}<0\qquad\text{for large~$n$}.
\end{equation}
Since $\ell_i\in J_\Delta$, we have $\Delta(\ell_i)<0$.
By \cref{thm:sgn:d<0:10}, there is an integer $M$ such that for all integers $n>M$,
the angle~$n\theta_{\ell_i}$ lies on the same side of the line~$\psi_{\ell_i}$.
This is impossible, because $\theta_{\ell_i}\in(0,\,\pi)$.
Hence, we have $\ell_i=x_\Delta^+$.

Now suppose that $f(x_\Delta^+)\ge 0$ and $i\ge1$.
From \cref{thm:RR:10}, we see that $R_n\subset J_0=J_\Delta$.
Thus we have $\xi_{n,\,n+1-i}<x_\Delta^+$ for large~$n$, which implies that $\ell_i\le x_\Delta^+$. 
Suppose, to the contrary, that $\ell_i<x_\Delta^+$. 
The polynomial $W_n(x)$ has exactly $i-1$ roots in the interval $(\ell_i,\,x_\Delta^+)$ for large $n$,
that is, the roots $\xi_{n,\,n+2-i}$, $\xi_{n,\,n+3-i}$, $\ldots$, $\xi_{n,\,n}$.
Thus, we get 
\[
W_n(\ell_i)W_n(x_\Delta^+)(-1)^{i-1}>0\qquad\text{for large $n$}.
\]
On the other hand, since $A(x_\Delta^+)>0$ and $f(x_\Delta^+)\ge 0$, 
we have $W_n(x_\Delta^+)>0$ for all $n$. Multiplying it with the above inequality, we obtain \cref{sgn:W:ell:10} again, 
which is absurd for the same reason.
This completes the proof.
\end{proof}

Applying the same idea to show that the number $x_\Delta^-$ is also a limit point, 
we find that the angles $n\theta_{\ell_i}$ reside on both sides of the line~$\psi_{\ell_i}$, but alternatively. 
This leads us to show that the alternation is impossible.
The next two lemmas serves for this aim, depending on the rationality of the number~$\theta_{x}$.
Let 
$\pi\Q=\{q\pi \colon q\in\Q\}$.

\begin{lem}\label{lem:Q:10}
Let $\theta=q\pi/p$, where $p$ is a positive integer, $q$ is an integer, and $(p,q)=1$.
Then the sequence $\{\pv(n\theta)\}_{n\ge1}$ is periodic,
with the minimum period 
\begin{equation}\label{def:p0}
p_0=\begin{cases}
p,&\text{if $q$ is even};\\[3pt]
2p,&\text{if $q$ is odd}.
\end{cases}
\end{equation}
Moreover, we have 
\begin{equation}\label{set:pv:10}
\{\pv(n\theta)\colon\,n\in[p_0]\}
\=\{2j\pi/p_0\colon\,j\in[\,0,\,p_0-1]\}.
\end{equation}
\end{lem}

\begin{proof}
See \cref{apd:Q}.
\end{proof}

\begin{lem}\label{lem:NQ:10}
Let $\theta\in\R\backslash\pi\Q$.
Then for any nonempty open interval $I\subset(0,\,2\pi)$, there exists an arbitrarily large integer~$m$ such that $\pv(m\theta)\in I$.
\end{lem}

\begin{proof}
See \cref{apd:NQ} for a proof by using Dirichlet's approximation theorem.
\end{proof}

By using the above two lemmas, we can show the impossibility of the aforementioned alternation.

\begin{lem}\label{lem:side:10}
Let~$x\in \R$ such that $\Delta(x)<0$. Let $M>0$.
Suppose that the lines $n\theta_{x}$ and $\psi_{x}$ do not coincide with each other for all $n>M$.
Then there exists an arbitrarily large integer~$n$ such that 
the angles~$n\theta_{x}$ and $(n+1)\theta_{x}$ lie on the same side of the line~$\psi_{x}$.
\end{lem}

\begin{proof}
Recall from \cref{ntt:theta:psi:10} that 
\begin{equation}\label[rl]{range:theta:psi:10}
\theta_x\in(0,\pi)
\rmand
\psi_x\in(0,\pi).
\end{equation} 

Assume that $\theta_{x}\notin\pi\Q$. Let $I=(\psi_x,\,\psi_x+\epsilon)$, where
\[
\epsilon=\min((\pi-\theta_{x})/2,\,\pi-\psi_x/2)\in(0,\,\pi/2).
\]
Using $\epsilon\le\pi-\psi_x/2$, we deduce that $I\subset(0,\,2\pi)$.
By \cref{lem:NQ:10} and using $\epsilon\le(\pi-\theta_{x})/2$, we infer that there is an arbitrarily large integer $n$ such that 
\[
\pv(n\theta_{x})\in I\subseteq(\psi_x,\,\psi_x+(\pi-\theta_{x})/2).
\]
Together with \cref{range:theta:psi:10},
we deduce that
\[
\pv(n\theta_{x})+\theta_{x}
\in(\psi_x,\,\psi_x+(\pi-\theta_{x})/2+\theta_{x})
\subset(\psi_x,\,\psi_x+\pi),
\]
which implies that the angle $(n+1)\theta_{x}$ lies to the left of the line $\psi_x$.
On the other hand, by using $\epsilon\le \pi-\psi_x/2$, we derive that 
$\pv(n\theta_{x})\in I\subset(\psi_x,\,\psi_x+\pi)$, which implies that the angle $n\theta_{x}$ also lies to the left of the line $\psi_x$.

Now suppose that $\theta_{x}\in\pi\Q$. By \cref{range:theta:psi:10}, we can denote $\theta_{x}=q\pi/p$,
where $p\in\Z^+$, $q\in[p-1]$, and $(p,q)=1$. It follows that 
\begin{equation}\label[ineq]{pf2:Q}
\pi/p_0+\theta_{x}
\=\pi/p_0+q\pi/p
\;\le\; 
\pi/p_0+(p-1)\pi/p
\;\le\; 
\pi.
\end{equation}
We will use the above inequality in the sequel.

In the polar coordinate system, the $p_0$ rays $2j\pi/p_0$ ($j\in[\,0,\,p_0-1]$)
partition the full circle equally into~$p_0$ angles of size $2\pi/p_0$.
Note that the lines $n\theta_{x}$ and $\psi_x$ do not coincide for $n>M$.
By \cref{set:pv:10}, the set $\{\pv(n\theta_x)\colon\,n\in\Z\}$ of lines is finite, which does not contain the line $\psi_x$.
Therefore, the minimum angle among the angles formed by the line $\phi_x$ and one of the above rays 
is of size at most a half of the size $2\pi/p_0$.
In other words, there is an integer $j_0\in [\,0,\,p_0-1)$ such that
\[
|2j_0\pi/p_0-\psi_x|\in(0,\,\pi/p_0].
\]
By \cref{set:pv:10}, we can suppose that $\pv(n_0\theta_{x})=2j_0\pi/p_0$, where $n_0\in[p_0]$.
Then the above range relation can be rewritten as 
\begin{equation}\label{pf1:range:10}
|\pv(n_0\theta_{x})-\psi_x|\in(0,\,\pi/p_0].
\end{equation}

If $\pv(n_0\theta_{x})-\psi_x>0$, then the above range relation gives that 
\[
\pv(n_0\theta_{x})\in(\psi_x,\,\psi_x+\pi/p_0]\subset(\psi_x,\,\psi_x+\pi),
\]
which implies that the angle $n_0\theta_{x}$ lies to the left of the line $\psi_x$.
By the above relation, and using \cref{range:theta:psi:10,pf2:Q}, we infer that
\[
\pv(n_0\theta_{x})+\theta_{x}
\in(\psi_x+\theta_x,\,\psi_x+\pi/p_0+\theta_{x}]
\subset(\psi_x,\,\psi_x+\pi],
\]
which implies that the angle $(n_0+1)\theta_x$ lies to the left of the line $\psi_x$,
or on the line $\psi_x$. The latter possibility never occurs
since the lines $\psi_x$ and $n\theta_x$ do not coincide for any $n$.
Now, by the periodicity obtained in \cref{lem:Q:10}, 
we see that there is an arbitrarily large integer~$n$ such that $\pv(n\theta_x)=\pv(n_0\theta_x)$,
and thus, both the angles~$n\theta_{x}$ and $(n+1)\theta_{x}$ lie to the left of the line~$\psi_{x}$.

Otherwise $\pv(n_0\theta_{x})-\psi_x<0$. Then \cref{pf1:range:10} gives that 
\begin{equation}\label[rl]{range2:pv:ntheta}
\pv(n_0\theta_{x})\in[\psi_x-\pi/p_0,\,\psi_x).
\end{equation}
Since $\pv(\psi)\in[\,0,\,2\pi)$ for all $\psi\in\R$, the above relation implies that $\pv(n_0\theta_{x})\in[\,0,\,\psi_x)$.
By \cref{def:pv}, we infer that the angle $n_0\theta_{x}$ lies to the right of the line $\psi_x$.
On the other hand, by \cref{range2:pv:ntheta,pf2:Q}, we infer that
\[
\pv(n_0\theta_{x})-\theta_x
\in[\psi_x-\pi/p_0-\theta_{x},\,\psi_x-\theta_{x})
\subset[\psi_x-\pi,\,\psi_x),
\]
which implies that the angle $(n_0-1)\theta_x$ lies to the right of the line $\psi_x$, or on the line $\psi_x$.
For the same reason, the second circumstance does not happen.
Hence, by the periodicity, there is an arbitrarily large integer~$n$ such that 
the angles~$n\theta_{x}$ and $(n+1)\theta_{x}$ lie to the right of the line~$\psi_{x}$.
This completes the proof.
\end{proof}

Now we are in a position to justify that the number $x_\Delta^-$ is a limit point.

\begin{thm}\label{thm:xd1:10}
Let $\{W_n(x)\}_{n\ge0}$ be a \oz, with ordered zero-set 
\[
\{\xi_{n,1},\,\xi_{n,2},\,\ldots,\,\xi_{n,n}\}.
\] 
Then we have 
$\xi_{n,i}\searrow x_\Delta^-$ for all $i\ge2$ if $f(x_\Delta^-)>0$, and for all $i\ge1$ otherwise.
\end{thm}

\begin{proof}
Let $i\ge 1$. From the property $R_{n+1}\bowtie R_n$ obtained in \cref{thm:RR:10},
we see that the sequence~$\xi_{n,i}$ ($n\ge i$) decreases as $n\to\infty$.
Since it is bounded by the interval $J_0$, the sequence $\xi_{n,i}$ converges.
Suppose that $\lim_{n\to\infty}\xi_{n,i}=\ell_i$.

Suppose that $f(x_\Delta^-)>0$ and $i\ge2$. 
From \cref{thm:bound:10}~(i), we see that $\xi_{n,i}>x_\Delta^-$ for $n\ge i$,
which implies that $\ell_i\ge x_\Delta^-$. 
Suppose, to the contrary, that $\ell_i>x_\Delta^-$. 
When $n$ is large, the polynomial $W_n(x)$ has exactly $i-2$ roots in the interval $(x_\Delta^-,\,\ell_i)$,
that is, the roots $\xi_{n,2},\,\xi_{n,3},\,\ldots,\,\xi_{n,\,i-1}$.
Thus, by the intermediate value theorem, we infer that 
\[
W_n(x_\Delta^-)W_n(\ell_i)(-1)^{i-2}\>0\qquad\text{for large $n$}.
\]
On the other hand, we have $A(x_\Delta^-)<0$ by \cref{sgn:A:xd:10}.
Since $f(x_\Delta^-)>0$, by \cref{W:xd}, we have 
$W_n(x_\Delta^-)(-1)^{n+1}>0$ for large $n$.
Multiplying it by the above inequality results in that 
\begin{equation}\label[ineq]{pf:sgn:W:ell:10}
W_n(\ell_{i})(-1)^{n+i}\<0\qquad\text{for large $n$}.
\end{equation}
To wit, the sign of the value $W_n(\ell_i)$ alternates as $n\to \infty$.
By \cref{thm:sgn:d<0:10}, the angle~$n\theta_{\ell_i}$ moves between the two sides of the line~$\psi_{\ell_i}$ alternatively 
for large~$n$. Since $\ell_i\in J_\Delta$, we infer that $\Delta(\ell_i)<0$, which 
contradicts \cref{lem:side:10}. This proves that $\ell_i=x_\Delta^-$.

Now, suppose that $f(x_\Delta^-)\le 0$, and that $i\ge 1$. 
By \cref{thm:RR:10}, we have $\ell_i\ge x_\Delta^-$.
Suppose, to the contrary, that $\ell_i>x_\Delta^-$. 
Along the same lines, we can show that 
\[
W_n(x_\Delta^-)W_n(\ell_i)(-1)^{i-1}\>0\qquad\text{for large $n$}.
\]
On the other hand, since $A(x_\Delta^-)<0$ and $f(x_\Delta^-)\le 0$, we have $W_n(x_\Delta^-)(-1)^{n}>0$ for large~$n$.
Multiplying it by the above inequality gives Ineq.~(\ref{pf:sgn:W:ell:10}),
which causes the same contradiction.
This completes the proof.
\end{proof}

Now we sum up the results obtained in this section to complete the proof of \cref{thm:10:normalized}.
By \cref{thm:RR:10}, every polynomial $W_n(x)$ is real-rooted and $R_{n+1}\bowtie R_n$.
By \cref{thm:xd2:10} and \cref{thm:xd1:10}, 
we have $\xi_{n,j}\searrow x_\Delta^-$ and $\xi_{n,\,n+1-j}\nearrow x_\Delta^+$ for all $j\ge 2$. 
As will be seen in the following rearrangement,
the limit of the smallest roots $\xi_{n,1}$ depends on the sign of the number $f(x_\Delta^-)$,
while the limit of the largest roots $\xi_{n,n}$ depends on the sign of the number $f(x_\Delta^+)$.
Recall from \cref{def:b0:10} that $b_0=|a-2|\sqrt{-d}$, and from \cref{sgn:f:xd1:10,sgn:f:xd2:10} that
\begin{align*}
&f(x_\Delta^-)>0
\eqrl
b<(a-2)\sqrt{-d},\rmand\\[4pt]
&f(x_\Delta^+)<0
\eqrl
b>(2-a)\sqrt{-d}.
\end{align*}
We are ready to state the remaining limits according to the ranges of the numbers $a$ and $b$.

\bigskip
\noindent{\bf (i)}
When $a\le 2$ and $b\le b_0$, we have $f(x_\Delta^-)\le0$ and $f(x_\Delta^+)<0$. 
Therefore, \cref{thm:xd1:10} gives that $\xi_{n,1}\searrow x_\Delta^-$, 
while \cref{thm:xd2:10} gives that $\xi_{n,n}\nearrow x_\Delta^+$.

\medskip
\noindent{\bf (ii)}
When $a>2$ and $b<b_0$, we have $f(x_\Delta^-)>0$ and $f(x_\Delta^+)<0$. 
Therefore, \cref{thm:xg1:10} gives that $\xi_{n,1}\searrow x_g^-$, 
while \cref{thm:xg2:10} gives that $\xi_{n,n}\nearrow x_g^+$.

\medskip
\noindent{\bf (iii)}
Otherwise, we have $b>b_0$ or ``$b=b_0$ and $a>2$''. 
In either case, we have $f(x_\Delta^-)\le0$ and $f(x_\Delta^+)<0$. 
Therefore, \cref{thm:xd1:10} gives that $\xi_{n,1}\searrow x_\Delta^-$, 
and \cref{thm:xg2:10} gives that $\xi_{n,n}\nearrow x_g^+$.

\bigskip
This completes the proof of \cref{thm:10:normalized}.

\section{\large Proof of \cref{thm:10}}\label{sec:pf2:10}

In this section we derive \cref{thm:10} by using \cref{thm:10:normalized}.
The proof is divided into two steps. First, we generalize \cref{thm:10:normalized} by dropping the restriction $b\ge0$; see \cref{prop:b<0:10}.
Second, we extend \cref{prop:b<0:10} by allowing the polynomial $W_1(x)=t(x-r)$, by translation and magnification.

\begin{prop}[allowing $b<0$]\label{prop:b<0:10}
Let $W_n(x)$ be polynomials defined by \cref{rec:10}, 
where $a>0$, $d<0$, and $b\in\R$, with $W_0(x)=1$ and $W_1(x)=x$.
Then every polynomial $W_n(x)$ is real-rooted. 
Denote the zero-set of the polynomial~$W_n(x)$ by~$R_n=\{\xi_{n,1},\,\xi_{n,2},\,\ldots,\,\xi_{n,n}\}$ in increasing order. 
Then we have $R_{n+1}\bowtie R_n$ and \cref{lim:10}.
Moreover, we have the following.
\begin{itemize}
\smallskip\item[(i)]
If $a\le 2$ and $|b|\le b_0$, then $\xi_{n,1}\searrow x_\Delta^-$ and $\xi_{n,n}\nearrow x_\Delta^+$.
\smallskip\item[(ii)]
If $a>2$ and $|b|<b_0$, then $\xi_{n,1}\searrow x_g^-$ and $\xi_{n,n}\nearrow x_g^+$.
\smallskip\item[(iii)]
Otherwise, we have $b\ne 0$, and the following.
\begin{itemize}
\smallskip\item[(iii)-1.]
If $b<0$, then $\xi_{n,1}\searrow x_g^-$ and $\xi_{n,n}\nearrow x_\Delta^+$.
\smallskip\item[(iii)-2.]
If $b>0$, then $\xi_{n,1}\nearrow x_\Delta^-$ and $\xi_{n,n}\nearrow x_g^+$.
\end{itemize} 
\end{itemize}
\end{prop}

\begin{proof}
See \cref{apd:b<0}.
\end{proof}

Now we are in a position to show \cref{thm:10}.
Suppose that all the hypotheses in \cref{thm:10} hold true.
Consider the sequence $\tW_n(x)$ defined by
\begin{equation}\label[def]{LinearTr:10}
\tW_n(x)=W_n(x/t+r).
\end{equation}
Replacing $x$ by $x/t+r$ in \cref{rec:10}, we obtain that 
\[
\tW_n(x)=(\ta x+\tb)\tW_{n-1}(x)+d\cdotp\tW_{n-2}(x),
\]
where $\ta=a/t$ and $\tb=A(r)$, 
with $\tW_0(x)=1$, $\tW_1(x)=x$. It follows that $\ta>0$.
By \cref{prop:b<0:10}, every polynomial $\tW_n(x)$ is distinct-real-rooted.
From \cref{LinearTr:10}, we see that every polynomial $W_n(x)$ is real-rooted. Let 
\[
\tR_n=\{\txi_{n,1},\,\txi_{n,2},\,\ldots,\,\txi_{n,n}\}
\]
be the ordered zero-set of the polynomial $\tW_n(x)$. 
Then we have
\begin{equation}\label{txi:10}
\xi_{n,i}=\txi_{n,i}/t+r.
\end{equation}
It is clear that magnification and translation preserve the interlacing property.
Thus $\tR_{n+1}\bowtie\tR_n$ implies that $R_{n+1}\bowtie R_n$.

Regarding the numbers $x_\Delta^\pm=x_\Delta^\pm(a,b)$ as functions of $a$ and $b$,
we can define that $x_{\tD}^\pm=x_\Delta^\pm(\ta,\tb)$.
Similarly, we define $x_{\tg}^\pm=x_g^\pm(\ta,\tb)$, and $\tb_0=b_0(\ta)$ 
Then we have
\begin{align}
\label{tD:10}
x_{\tD}^\pm
&\={-\tb\pm2\sqrt{-d}\over \ta}
=\frac{-A(r)\pm2\sqrt{-d}}{a/t},\rmand\\[5pt]
x_{\tg}^\pm
&\=\begin{cases}
\displaystyle{-\tb\pm\sqrt{\tb^2-4d(\ta-1)}\over 2(\ta-1)},&\text{if $\ta\ne1$}\\[5pt]
\displaystyle{-{d\over\tb}},&\text{if $\ta=1$ and $\tb\ne0$}
\end{cases}\notag\\[5pt]
\label{tg:10}
&\=\begin{cases}
\displaystyle{\frac{-A(r)\pm\sqrt{A(r)^2-4d(a/t-1)}}{2(a/t-1)}},&\text{if $a\ne t$}\\[8pt]
\displaystyle{-{d\over A(r)},}&\text{if $a=t$ and $A(r)\ne 0$},
\end{cases}\\[5pt]
\label{tb0:10}
\tb_0&\=|\ta-2|\sqrt{-d}=\frac{|a-2t|\sqrt{-d}}{t}.
\end{align}

By \cref{prop:b<0:10}, we have 
\[
\txi_{n,j}\searrow x_{\tD}^-
\rmand
\txi_{n,\,n+1-j}\nearrow x_{\tD}^+
\all{j\ge2}.
\]
By \cref{txi:10,tD:10}, the above relations can be rewritten as
\[
t(\xi_{n,j}-r)\searrow \frac{-A(r)-2\sqrt{-d}}{a/t}
\rmand
t(\xi_{n,\,n+1-j}-r)\nearrow \frac{-A(r)+2\sqrt{-d}}{a/t},
\]
namely, \cref{lim:10} got proved.
The remaining cases are shown individually below.

\bigskip
\noindent{\bf (i)}
When $a\le 2t$, we have $\ta\le 2$. Suppose that $r\in[r^-,r^+]$. Then $(r-r^-)(r-r^+)\le 0$.
Substituting \cref{def:r+-:10} into it gives that
\[
t^2A(r)^2+d(a-2t)^2\le 0.
\]
By \cref{tb0:10}, the above inequality is equivalent to that $|\tb|\le \tb_0$. By \cref{prop:b<0:10}, we have
\[
\txi_{n,1}\searrow x_{\tD}^-
\rmand
\txi_{n,n}\nearrow x_{\tD}^+.
\]
Along the same lines for proving \cref{lim:10}, we find $\xi_{n,1}\searrow x_\Delta^-$ and $\xi_{n,n}\nearrow x_\Delta^+$.

\medskip
\noindent{\bf (ii)}
When $a>2t$, we have $\ta>2$. Suppose that $r\in(r^-,r^+)$. Same to the above proof for (i), we obtain that $|\tb|<\tb_0$.
By \cref{prop:b<0:10}, we have
\[
\txi_{n,1}\searrow x_{\tg}^-
\rmand 
\txi_{n,n}\nearrow x_{\tg}^+.
\]
By \cref{txi:10,tg:10}, the above relations can be recast as
\begin{align*}
t(\xi_{n,1}-r)&\ \searrow\ \frac{-A(r)-\sqrt{A(r)^2-4d(a/t-1)}}{2(a/t-1)}
\rmand\\
t(\xi_{n,n}-r)&\ \nearrow\ \frac{-A(r)+\sqrt{A(r)^2-4d(a/t-1)}}{2(a/t-1)}.
\end{align*}
In view of \cref{def:xi+-:10}, the above relations reduce to $\xi_{n,1}\searrow \xi^-$ and $\xi_{n,n}\nearrow \xi^+$ respectively.

\medskip
\noindent{\bf (iii)}
Suppose that ``$r=r^-$ and $a>2t$'', or $r<r^-$.
Then we have $r\le r^-$.
By using \cref{def:r+-:10}, we infer that 
\[
\tb=A(r)\le -|a-2t|\sqrt{-d}/t\le 0.
\]
Assume that $\tb=0$. Then we have $a=2t$ and thus get into Case (i), a contradiction.
Therefore, $\tb<0$.
By \cref{prop:b<0:10} (iii)-1, we have
\[
\txi_{n,1}\searrow x_{\tg}^-
\rmand
\txi_{n,n}\nearrow x_{\tD}^+.
\]
By \cref{txi:10,tg:10,tD:10}, the above relations can be rewritten as
\begin{align*}
t(\xi_{n,1}-r)&\ \searrow\ \frac{-A(r)-\sqrt{A(r)^2-4d(a/t-1)}}{2(a/t-1)}
\rmand\\
t(\xi_{n,n}-r)&\ \nearrow\ \frac{-A(r)\pm2\sqrt{-d}}{a/t}.
\end{align*}
Same to the proofs of (i) and (ii), we deduce that $\xi_{n,1}\searrow \xi^-$ and $\xi_{n,n}\nearrow x_\Delta^+$.

\medskip
\noindent{\bf (iv)} It is highly similar to Case (iii), and we omit it.
This completes the proof of \cref{thm:10}.

\section{\large Concluding Remarks}

This section explains why we set $a>0$, $d<0$, and $t>0$ in \cref{thm:10}.

First of all, the restriction $adt\ne0$ is without loss of generality.
In fact,
the sign restriction $a\ne 0$ comes from the type $(1,0)$ of recursive polynomials as the topic of this paper.
When $d=0$, it is clear that $W_n(x)=A(x)^{n-1}W_1(x)$ by \cref{rec:10}, and the root geometry problem becomes trivial.
When $t=0$, the polynomial $W_1(x)=0$,
and one may consider the sequence $\{W_{n+2}(x)/d\}_{n\ge0}$.

Second, with the assumption $adt\ne0$ in hand, the real-rootedness for every polynomial $W_n(x)$ is still not true in general.
In fact, consider the case $b=0$, $t=1$, $W_0(x)=1$ and $W_1(x)=x$.
\begin{itemize}
\item
When $ad>0$, the polynomial $W_2(x)=ax^2+d$ has no real roots since its sign is same to the sign of the number $d$.
\item
When $a<0<d$, the polynomial $W_3(x)=x(a^2x^2+d(a+1))$ has non-real roots as if $a+1>0$.
\end{itemize}

The remaining case except that we handled in this paper is that all the parameters~$a$, $d$, and~$t$ are negative. 
In this case, by considering the sequence $\{(-1)^nW_n(r-x/t)\}_{n\ge0}$ with the aid of \cref{thm:10}, one may derive the real-rootedness of every polynomial $W_n(x)$,
as well as the interlacing property and the limit points of the union of the zero-sets. 
Yet in applications, it is not often to meet such situation that all $a$, $d$, and $t$ are negative.

\appendix

\section{Proof of \cref{lem:xg:10}}\label{apd:xg:10}

Let $x\in\{x_g^\pm\}$. Then we have $g(x)=0$. From \cref{def:g:10} of the function~$g(x)$ as $g^-(x)g^+(x)$, 
we have either $g^-(x)=0$ or $g^+(x)=0$.
To show that $W_n(x)=x^n$, we split into two cases.

Assume that $g^-(x)=0$. 
From \cref{def:g+-:10} of the function $g^-(x)$ as $\bigl(f(x)-\sqrt{\Delta(x)}\,\bigr)/2$, 
and from \cref{def:f:10} of the function $f(x)$ as $2x-A(x)$,  we have 
\begin{equation}\label{pfA1:10}
A(x)+\sqrt{\Delta(x)}\=2x.
\end{equation}
It follows from \cref{def:g+-:10} of the function $g^+(x)$ that 
\begin{equation}\label{pfA2:10}
g^+(x)\=\frac{2x-A(x)+\sqrt{\Delta(x)}}{2}\=\sqrt{\Delta(x)}.
\end{equation}
If $\Delta(x)\ne0$, then \cref{lem:00} gives that 
\begin{equation}\label{pfA3:10}
W_n(x)
\={g^+(x)(A(x)+\sqrt{\Delta(x)}\,)^n-g^-(x)(A(x)-\sqrt{\Delta(x)}\,)^n\over 2^{n}\sqrt{\Delta(x)}}.
\end{equation}
Substituting the condition $g^-(x)=0$, \cref{pfA1:10,pfA2:10} into the above equation, we obtain that 
\[
W_n(x)
\={\sqrt{\Delta(x)}\,(2x)^n-0\over 2^{n}\sqrt{\Delta(x)}}
\=x^n.
\]
Otherwise $\Delta(x)=0$, then \cref{pfA1:10} reduces to $A(x)=2x$. Since $W_1(x)=x$, by \cref{lem:00}, we deduce that
\[
W_n(x)
\=\bgg{1+{n(2W_1(x)-A(x))\over A(x)}}\bgg{{A(x)\over 2}}^n
\=x^n.
\]

When $g^+(x)=0$, 
from \cref{def:g+-:10} of the function $g^+(x)$ as $\bigl(f(x)+\sqrt{\Delta(x)}\,\bigr)/2$, 
and from \cref{def:f:10},  we have 
\begin{equation}\label{pfA1':10}
A(x)-\sqrt{\Delta(x)}\=2x.
\end{equation}
It follows from \cref{def:g+-:10} that 
\begin{equation}\label{pfA2':10}
g^-(x)\=\frac{2x-A(x)-\sqrt{\Delta(x)}}{2}\=-\sqrt{\Delta(x)}.
\end{equation}
If $\Delta(x)\ne0$, 
substituting the condition $g^+(x)=0$, \cref{pfA1':10,pfA2':10} into \cref{pfA3:10}, we obtain that 
\[
W_n(x)
\={0-(-\sqrt{\Delta(x)}\,)(2x)^n\over 2^{n}\sqrt{\Delta(x)}}
\=x^n.
\]
This completes the proof of the identity $W_n(x_g^\pm)=(x_g^\pm)^n$.

Below we show the results (i) and (ii) in \cref{lem:xg:10}.

\medskip
\noindent (i) Suppose that $a\le 1$. 
If $a<1$, from \cref{def:xg+-:10} of the numbers $x_g^\pm$, 
the desired inequalities $0<x_g^+\le x_g^-$ can be rewritten as 
\[
0\<{-b+\sqrt{b^2-4d(a-1)}\over 2(a-1)}
\;\le\;
{-b-\sqrt{b^2-4d(a-1)}\over 2(a-1)}.
\]
Since $a<1$, the first inequality holds because $d<0$, and the second inequality holds trivially.
Otherwise $a=1$ and $b\ne0$. 
By \cref{def:xg+-:10}, the desired inequalities $0<x_g^+\le x_g^-$ become 
\[
0<-d/b\le -d/b,
\] 
which is also trivial since $d<0<b$.
The remaining inequalities $W_n(x_g^\pm)>0$ follow immediately from the equations $W_n(x_g^\pm)=x_g^\pm$ and the inequalities $x_g^\pm>0$.

\smallskip
\noindent (ii) Suppose that $a>1$.
From \cref{def:xg+-:10} of the numbers $x_g^\pm$, 
the desired inequalities $x_g^-<0<x_g^+$ can be rewritten as 
\[
{-b-\sqrt{b^2-4d(a-1)}\over 2(a-1)}\<0\<{-b+\sqrt{b^2-4d(a-1)}\over 2(a-1)}.
\]
Both of them hold trivially since $d<0\le b$.
Consequently, since $W_n(x_g^\pm)=(x_g^\pm)^n$, we have
\begin{eqnarray*}
W_n(x_g^-)(-1)^n&=&(x_g^-)^n(-1)^n>0\rmand\\
W_n(x_g^+)&=&(x_g^+)^n>0.
\end{eqnarray*}
This completes the proof of \cref{lem:xg:10}.

\bigskip

\section{Proof of \cref{lem:J0:10}}\label{apd:J0:10}

Suppose that 
\begin{equation}\label[ineq]{pf18}
f(x_\Delta^+)<0.
\end{equation}
Substituting \cref{f:xd2:10} into \cref{pf18}, and canceling the positive factor $2/a$, we obtain that 
\begin{equation}\label[ineq]{pf19}
b>(2-a)\sqrt{-d}.
\end{equation}
When $a=1$, \cref{pf19} reduces to $b>\sqrt{-d}$, which implies that $b\neq0$.
It follows from \cref{def:xg+-:10} that the numbers $x_g^\pm$ are well-defined.
We shall show that 
\begin{itemize}
\smallskip\item[(i)]
the numbers $x_g^\pm$ are real; 
\smallskip\item[(ii)]
we have $x_g^+>x_\Delta^+$;
\smallskip\item[(iii)]
if additionally we have $f(x_\Delta^-)>0$, then $x_g^-<x_\Delta^-$.
\end{itemize}
For (i) and (ii), we split into the cases $a=1$ and $a\ne1$. 
The desired result (iii) will be shown individually.

Assume that $a=1$.
From \cref{def:xg+-:10} of the numbers~$x_g^\pm$ as the real number $-d/b$ identically,
and from \cref{def:xd+-:10} of the number~$x_\Delta^+$ as $-b+2\sqrt{-d}$,
it is routine to compute that
\[
x_g^+-x_\Delta^+
\=-\frac{d}{b}-(-b+2\sqrt{-d})
\=\frac{(b-\sqrt{-d}\,)^2}{b}
\>0.
\]
In other words, the desired results~(i) and~(ii) are true when $a=1$.

Suppose that $a\ne1$. From \cref{def:xg+-:10}, we have
\[
x_g^\pm={-b\pm\sqrt{\Delta_g}\over 2(a-1)},
\qquad\text{where $\Delta_g=b^2-4d(a-1)$}.
\]
To show the realness of the numbers $x_g^\pm$, it suffices to prove that $\Delta_g>0$. 
In fact, when $a>1$, one has $\Delta_g>0$ since $d<0$.
When $0<a<1$, squaring both sides of \cref{pf19} gives that
$b^2>-d(2-a)^2$. Therefore, we infer that
\[
\Delta_g\=b^2-4d(a-1)\>-d(2-a)^2-4d(a-1)\=-da^2\>0.
\]
This confirms the desired result (i).

In order to show the desired result~(ii) for $a\ne1$, as well as the desired result~(iii), we make some preparations uniformly.

\Cref{def:g:10,def:g+-:10} imply that
\begin{equation}\label[ineq]{pf:eq:dd:xg:10}
\Delta(x_g^\pm)\=f^2(x_g^\pm)-4g(x_g^\pm)\=f^2(x_g^\pm)\;\ge\;0.
\end{equation}
Since the function $\Delta(x)$ is a quadratic polynomial with positive leading coefficient,
whose roots are $x_\Delta^\pm$ with $x_\Delta^-<x_\Delta^+$,
we deduce that 
\begin{align}
\label[rl]{pf20}
&\text{either}\quad x_g^-\le x_\Delta^-\quad \text{or}\quad x_g^-\ge x_\Delta^+,\qquad\text{and that}\\
\label[rl]{pf17}
&\text{either}\quad x_g^+\le x_\Delta^-\quad \text{or}\quad x_g^+\ge x_\Delta^+. 
\end{align}
On the other hand, let $X=(a-2)b$ and $Y=a\sqrt{\Delta_g}$. Since $a>0$, we have
\begin{equation}\label[ineq]{pf:Y>0:10}
Y>0.
\end{equation}
By \cref{def:xg+-:10} of the numbers~$x_g^\pm$, it is direct to calculate that
\begin{align}
\label{pf:xA-xg:10}
-{b\over a}-x_g^\pm&\={X\pm Y\over 2a(1-a)},\rmand\\
\label{pf:X2-Y2:10}
X^2-Y^2&\=4(1-a)(b^2-a^2d).
\end{align}
Since $d<0<a$, we have $b^2-a^2d>0$. Since $a\ne 1$, \cref{pf:X2-Y2:10} implies that
\begin{equation}\label[ineq]{pf:sgn:X-Y:10}
(X+Y)(X-Y)(1-a)>0.
\end{equation}

Now we are in a position to show the desired result (ii) for $a\ne1$.
Suppose, to the contrary, that $x_g^+\le x_\Delta^+$. 
Then, by Relation (\ref{pf17}), we deduce that either $x_g^+=x_\Delta^+$, or 
\[
x_g^+\;\le\; x_\Delta^-\=\frac{-b-2\sqrt{-d}}{a}\<-\frac{b}{a}.
\]  
In the former case, we infer from \cref{pf:eq:dd:xg:10} that
\[
f^2(x_\Delta^+)\=f^2(x_g^+)\=\Delta(x_g^+)\=\Delta(x_\Delta^+)\=0.
\]
It follows that $f(x_\Delta^+)=0$, contradicting \cref{pf18}.
In the latter case,
\cref{pf:xA-xg:10} implies that
\begin{equation}\label[ineq]{pf:sgn:X+Y:10}
(X+Y)(1-a)>0.
\end{equation}
Multiplying \cref{pf:sgn:X+Y:10} by \cref{pf:sgn:X-Y:10}, we find that $X-Y>0$. 
Together with \cref{pf:Y>0:10}, we conclude that 
\[
X>Y>0.
\]
On one hand, since $X=(a-2)b$ and $X>0$, we obtain that $a>2$.
On the other hand, since $X>0$ and $Y>0$, from \cref{pf:sgn:X+Y:10}, we deduce that $a<1$,
a contradiction. This proves~(ii).

At last, let us show the desired result~(iii). Suppose that 
\begin{equation}\label[ineq]{pf27}
f(x_\Delta^-)\>0.
\end{equation}
By \cref{f:xd1:10}, \cref{pf27} can be rewritten as $b<(a-2)\sqrt{-d}$.
Since $b\ge0$, we have 
\[
a>2.
\]
Suppose, by way of contradiction, that $x_g^-\ge x_\Delta^-$.
Then, by \cref{pf20}, we deduce that either $x_g^-=x_\Delta^-$, or 
\[
x_g^-\;\ge\; x_\Delta^+\=\frac{-b+2\sqrt{-d}}{a}\>-\frac{b}{a}.
\]
In the former case, we infer from \cref{pf:eq:dd:xg:10} that 
\[
f^2(x_\Delta^-)\=f^2(x_g^-)\=\Delta(x_g^-)\=\Delta(x_\Delta^-)\=0.
\]
It follows that $f(x_\Delta^-)=0$, contradicting \cref{pf27}.
In the latter case, \cref{pf:xA-xg:10} implies that
$(X-Y)(1-a)<0$. Since $a>2$, we infer that $X>Y$.
Together with \cref{pf:Y>0:10}, we conclude that 
\[
X>Y>0,
\]
which contradicts \cref{pf:sgn:X-Y:10}.
This completes the proof of \cref{lem:J0:10}.

\bigskip

\section{Proof of \cref{lem:f:xg2:10}}\label{apd:f:10}

Suppose that $f(x_\Delta^+)<0$. Then we have \cref{pf19}.

When $a\ge2$, from \cref{def:f:10}, the function $f(x)=(2-a)x-b$ is decreasing.
Together with \cref{pf18}, we infer that $f(x)<0$ for all $x>x_\Delta^+$.

Let $0<a<2$ below. By the monotonicity of the function~$f(x)$, it suffices to show that $f(x_g^+)<0$.

When $a=1$, \cref{pf19} reduces to $b>\sqrt{-d}$, that is, $b^2+d>0$. 
On the other hand, from \cref{def:xg+-:10} of the number~$x_g^+$, we have $x_g^+=-d/b$. It follows that
\[
f(x_g^+)=x_g^+-b=-{d+b^2\over b}<0.
\]

When $a\ne1$, recall that $\Delta_g=b^2-4d(a-1)$ is the discriminant of the quadratic function polynomial~$g(x)$.
By \cref{def:xg+-:10} of the number~$x_g^+$, it is direct to compute that 
\[
2(1-a)f(x_g^+)=(a-2)\sqrt{\Delta_g}+ab,
\]
which implies the equivalence relation
\[
(1-a)f(x_g^+)>0
\eqrl
ab>(2-a)\sqrt{\Delta_g}.
\]
Since $a<2$, squaring both sides of the above rightmost inequality gives another equivalence relation:
\[
ab>(2-a)\sqrt{\Delta_g}
\eqrl
(1-a)(b^2+d(a-2)^2)<0.
\]
By \cref{pf19}, we have $b^2+d(a-2)^2>0$. 
Therefore, transiting the above two equivalence relations gives the following equivalence relation:
\[
(1-a)f(x_g^+)>0
\eqrl
1-a<0.
\]
Hence, we infer that $f(x_g^+)<0$. This completes the proof.

\section{Proof of \cref{lem:Q:10}}\label{apd:Q}
Let $p_0$ be the integer defined by \cref{def:p0}. 
From \cref{def:pv}, we have $\pv(\theta+2k\pi)=\pv(\theta)$ for any $\theta\in\R$ and $k\in\Z$. Thus, we have 
\[
\pv((p_0+j)\theta)=\pv(p_0q\pi/p+j\theta)=\pv(j\theta).
\]
In other words, the sequence $\{\pv(n\theta)\}_{n\ge1}$ is periodic with a period $p_0$.

To show the minimality of the period $p_0$, 
it suffices to show that the numbers $\pv(n\theta)$ for $n\in[p_0]$ are pairwise distinct.
In fact, otherwise, there are three integers $j_1$, $j_2$, and $k$, 
such that $1\le j_1<j_2\le p_0$, and that $(j_2-j_1)\theta=2k\pi$, that is, $(j_2-j_1)q=2kp_0$.
Since $1\le j_2-j_1\le p_0-1$ and $(p_0,q)=1$, 
the left hand side $(j_2-j_1)q$ is not divided by the factor~$p_0$ of the right hand side, which is absurd.

Let $L$ (resp., $R$) denote the set on the left (resp., right) hand side of \cref{set:pv:10}.
We have just proved that the set $L$ has cardinality $p_0$, as same as the set $R$.
Therefore, to justify $L=R$ as sets, it suffices to show that the set~$L$ is contained in the set~$R$.
In fact, we have 
\[
\pv(n\theta)=\pv(nq\pi/p)=\begin{cases}
\pv(2(nq/2)\pi/p_0),&\text{if $q$ is even};\\[5pt]
\pv(2nq\pi/p_0),&\text{if $q$ is odd}.
\end{cases}
\]
Since $\pv(\psi)\in[\,0,\,2\pi)$ for any angle $\psi$, the rightmost expression in the above formula implies that the number $\pv(n\theta)$ has the form $2j\pi/p_0$,
where $j\in[\,0,\,p_0-1]$. Namely, we have $L\subseteq R$.
This completes the proof.

\section{Proof of \cref{lem:NQ:10}}\label{apd:NQ}

We will need Dirichlet's approximation theorem.

\begin{thm}[Dirichlet's approximation theorem]
For any real number $\mu$ and any positive integer~$N$,
there exist integers~$p$ and~$q$ such that $q\in[N]$ and $|q\mu-p|\le 1/(N+1)$.
\end{thm}

For the sake of proving \cref{lem:NQ:10}, we show the following stronger proposition.

\begin{prop}\label{prop:NQ}
Let $\theta\in\R$ such that $\theta\not\in\pi\Q$.
Then for any positive integer $M$, and for any nonempty open interval $I\subset(0,\,2\pi)$,
there exist positive integers $m$ and $q$ such that 
\[
\pv((m+j)q\theta)\in I\all{j\in[M]}.
\]
\end{prop}

\begin{proof}
Let $M$ be a positive integer.
Let $I=(u,v)$ be an open interval such that $0<u<v<2\pi$.
Let 
\[
\epsilon={v-u\over M+1}.
\]
Take an integer $N$ such that $2\pi/(N+1)<\epsilon$.
By Dirichlet's approximation theorem,
there exist integers $p$ and $q'$ such that $q'\in[N]$ and $|q'\theta/(2\pi)-p|<1/(N+1)$.
Therefore, we have
\[
|q'\theta-2p\pi|<{2\pi\over N+1}<\epsilon.
\]
In other words, we have either $\pv(q'\theta)\in[\,0,\epsilon)$ or $2\pi-\pv(q'\theta)\in[\,0,\epsilon)$.
Since $\theta\not\in\pi\Q$, the principal value $\pv(q'\theta)$ does not vanish, that is,
either $\pv(q'\theta)\in(0,\,\epsilon)$ or $\pv(q'\theta)\in(2\pi-\epsilon,\,2\pi)$.
For the latter case, there is a positive integer~$k$ such that $\pv(kq'\theta)\in(0,\,\epsilon)$, and we define $q=kq'$.
In the former case, we define $q=q'$. 
In summary, we obtain a positive integer~$q$ such that $\pv(q\theta)\in(0,\,\epsilon)$.

Now, let $m$ be the minimum positive integer such that
\begin{equation}\label[ineq]{pf:eq4}
(m+1)\pv(q\theta)\>u.
\end{equation}
Since $\pv(q\theta)<\epsilon=(v-u)/(M+1)$, we have $M\pv(q\theta)<v-u-\pv(q\theta)$.
It follows that 
\begin{equation}\label[ineq]{pf:eq5}
(m+M)\pv(q\theta)\<u+v-u-\pv(q\theta)\<v.
\end{equation}
Hence, by \cref{pf:eq4,pf:eq5}, we conclude that 
\[
u\<(m+1)\pv(q\theta)\<(m+2)\pv(q\theta)\<\cdots\<(m+M)\pv(q\theta)\<v.
\]
Let $j\in[M]$. Since $(m+j)\pv(q\theta)\in(u,v)\subset(0,2\pi)$,
we infer that $\pv((m+j)q\theta)=(m+j)\pv(q\theta)$, which implies the desired relation.
This completes the proof of \cref{prop:NQ}.
\end{proof}

\section{Proof of \cref{prop:b<0:10}}\label{apd:b<0}

Since Case (iii) is the exclusive case of (i) and (ii), we have either $|b|>b_0$, or ``$|b|=b_0$ and $a>2$''.
Assume that $b=0$, then the former possibility is impossible, and the latter possibility $|b|=b_0$ implies that $b_0=0$ and thus $a=2$,
a contradiction. Thus we have $b\ne 0$ for Case (iii).

The results for $b\ge 0$ are exactly those in \cref{thm:10:normalized}. 
Let $b<0$. Define 
\begin{equation}\label[def]{def:tW}
\tW_n(x)=(-1)^nW_n(-x).
\end{equation}
The the polynomials $\tW_n(x)$ satisfy the recursion
\[
\tW_n(x)=(ax-b)\tW_{n-1}(x)+d\tW_{n-2}(x).
\]
with same initiations $\tW_0(x)=1$ and $\tW_1(x)=x$. 
By \cref{thm:10:normalized}, every polynomial $\tW_n(x)$ is distinct-real-rooted.
Let 
\[
\tR_n=\{\txi_{n,1},\,\txi_{n,2},\,\ldots,\,\txi_{n,n}\}
\]
be the ordered zero-set of the polynomial $\tW_n(x)$. 
From \cref{def:tW}, we see that
the roots of the polynomial $W_n(x)$ are the opposites of the roots of the polynomial $\tW_n(x)$.
It follows that every polynomial $W_n(x)$ is real-rooted, with the ordered zero-set 
\[
R_n=\{-\txi_{n,n},\,-\txi_{n,\,n-1},\,\ldots,\,-\txi_{n,1}\}.
\]
In other words, we have 
\begin{equation}\label{txi:b<0:10}
\xi_{n,i}=-\txi_{n,\,n+1-i},\all{i\in[n]}.
\end{equation}
Since $\tR_{n+1}\bowtie \tR_n$, we infer that $R_{n+1}\bowtie R_n$.

Regarding the numbers $x_\Delta^\pm=x_\Delta^\pm(b)$ as a functions in~$b$,
we can define the numbers $x_{\tD}^\pm=x_\Delta^\pm(-b)$.
Similarly, we define $x_{\tg}^\pm=x_g^\pm(-b)$. 
Then we have
\begin{align}
\label{tD:b<0:10}
x_{\tD}^\pm
&\=\frac{-(-b)\pm2\sqrt{-d}}{a}
=-\frac{-b\mp2\sqrt{-d}}{a}
=-x_\Delta^\mp,\rmand\\[5pt]
\notag
x_{\tg}^\pm
&\=\begin{cases}
\displaystyle{-(-b)\pm\sqrt{(-b)^2-4d(a-1)}\over 2(a-1)},&\text{if $a\ne1$},\\[8pt]
\notag
\displaystyle{-{d\over -b},}&\text{if $a=1$ and $b\ne 0$},
\end{cases}\\[5pt]
&\=\begin{cases}
\displaystyle{-\frac{-b\mp\sqrt{b^2-4d(a-1)}}{2(a-1)}},&\text{if $a\ne1$},\\[8pt]
\notag
\displaystyle{-\bgg{-{d\over b}},}&\text{if $a=1$ and $b\ne 0$},
\end{cases}\\[5pt]
\label{tg:b<0:10}
&\=-x_g^\mp.
\end{align}
By \cref{thm:10:normalized}, we have 
\[
\txi_{n,j}\searrow x_{\tD}^-
\rmand
\txi_{n,\,n+1-j}\nearrow x_{\tD}^+
\all{j\ge2}.
\]
By \cref{txi:b<0:10,tD:b<0:10}, the above relations can be rewritten as
\[
-\xi_{n,\,n+1-j}\searrow -x_\Delta^+
\rmand
-\xi_{n,\,j}\nearrow -x_\Delta^-,
\]
namely, Relation~\eqref{lim:10} got proved.
The remaining cases are shown individually below.

\bigskip
\noindent{\bf (i)}
When $a\le 2$ and $|b|\le b_0$, we have $-b\le b_0$.
By \cref{thm:10:normalized} (i), we have 
\[
\txi_{n,1}\searrow x_{\tD}^-
\rmand
\txi_{n,n}\nearrow x_{\tD}^+.
\]
By using \cref{txi:b<0:10,tD:b<0:10}, the above limits can be rewritten as
\[
-\xi_{n,n}\searrow -x_\Delta^+
\rmand
-\xi_{n,1}\nearrow -x_\Delta^-,
\]
that is the desired limits $\xi_{n,1}\searrow x_\Delta^-$ and $\xi_{n,n}\nearrow x_\Delta^+$.

\medskip
\noindent{\bf (ii)}
When $a>2$ and $|b|<b_0$, we have $-b<b_0$.
By \cref{thm:10:normalized} (ii), we have 
\[
\txi_{n,1}\searrow x_{\tg}^-
\rmand
\txi_{n,n}\nearrow x_{\tg}^+.
\]
By using \cref{txi:b<0:10,tg:b<0:10}, the above limits can be rewritten as
\[
-\xi_{n,n}\searrow -x_g^+
\rmand
-\xi_{n,1}\nearrow -x_g^-,
\]
that is the desired limits $\xi_{n,1}\searrow x_g^-$ and $\xi_{n,n}\nearrow x_g^+$.

\medskip
\noindent{\bf (iii)}
Otherwise, we have either $|b|>b_0$ or ``$|b|=b_0$ and $a>2$''.
Since $b<0$, we only need to show (iii)-1. In this case, we have either $-b>b_0$ or ``$-b=b_0$ and $a>2$''.
By \cref{thm:10:normalized} (iii), we have 
\[
\txi_{n,1}\searrow x_{\tD}^-
\rmand
\txi_{n,n}\nearrow x_{\tg}^+.
\]
By using \cref{txi:b<0:10,tD:b<0:10,tg:b<0:10}, the above limits can be rewritten as
\[
-\xi_{n,n}\searrow -x_\Delta^+
\rmand
-\xi_{n,1}\nearrow -x_g^-,
\]
that is the desired limits $\xi_{n,1}\searrow x_g^-$ and $\xi_{n,n}\nearrow x_\Delta^+$.
This completes the proof of \cref{prop:b<0:10}.

\end{document}